\documentclass[a4paper,10pt]{scrartcl}
\usepackage[utf8]{inputenc}
\usepackage[english]{babel}
\usepackage[fixlanguage]{babelbib}
\usepackage{cite}
\usepackage{amssymb, amsmath, amstext, amsopn, amsthm, amscd, amsxtra, amsfonts}
\usepackage{yfonts, mathrsfs}
\usepackage{colonequals}
\usepackage{enumerate}
\usepackage{graphicx}%
\usepackage{colonequals}
\usepackage[all]{xy}

\usepackage{hyperref}%
\hypersetup{
urlcolor = red,
colorlinks = true,
linkcolor = blue,
citecolor = blue,
linktocpage = true,
pdftitle = {A differentiable monoid of smooth maps on Lie groupoids},
pdfauthor = {Habib Amiri, Alexander Schmeding},
bookmarksopen = true,
bookmarksopenlevel = 1,
unicode = true,
hypertexnames =false
}%

\swapnumbers
\newtheoremstyle{standard}
{16pt} %Abstand nach oben
{16pt} %Abstand nach oben
{} %Schrift Text
{} %Einrücken
{\bfseries}%\normalfont} %Schrift Kopf
{} %Zeichen nach Kopf
{ } %Abstand
{{\thmname{#1~}}{\thmnumber{#2.}}\thmnote{~(#3)}} %Format

\newtheoremstyle{kursiv}
{16pt} %Abstand nach oben
{16pt} %Abstand nach oben
{\itshape} %Schrift Text
{} %Einrücken
{\bfseries}%\normalfont} %Schrift Kopf
{} %Zeichen nach Kopf
{ } %Abstand
{{\thmname{#1~}}{\thmnumber{#2.}}\thmnote{~(#3)}} %Format

\theoremstyle{standard}
\newtheorem{defn} [subsection]{Definition}
\newtheorem{ex} [subsection]{Example}

\newtheorem{ques} [subsection]{Question}
\newtheorem{rem} [subsection]{Remark}

\newtheorem{setup} [subsection]{}

\theoremstyle{kursiv}
\newtheorem{thm}[subsection]{Theorem}
\newtheorem{prop} [subsection]{Proposition}
\newtheorem{cor} [subsection]{Corollary}
\newtheorem{lem} [subsection]{Lemma}

\newcommand{\cD}{\ensuremath{\mathcal{D}}}

\newcommand{\cF}{\ensuremath{\mathcal{F}}}
\newcommand{\cG}{\ensuremath{\mathcal{G}}}

\newcommand{\Lf}{\ensuremath{\mathbf{L}}}

\newcommand{\dd}{\mathop{}\!\mathrm{d}}

\newcommand{\R}{\ensuremath{\mathbb{R}}}

\newcommand{\N}{\ensuremath{\mathbb{N}}}

\newcommand{\toto}{\ensuremath{\nobreak\rightrightarrows\nobreak}}
\newcommand{\coloneq}{\colonequals}
\newcommand{\SG}{S_{\cG}}

\newcommand{\Gtwo}{\ensuremath{G^{(2)}}}

\newcommand{\Prop}{\ensuremath{\operatorname{Prop}}}
\newcommand{\Bis}{\ensuremath{\operatorname{Bis}}}
\newcommand{\ev}{\ensuremath{\operatorname{ev}}}
\DeclareMathOperator{\one}{\mathbf{1}}
\DeclareMathOperator{\supp}{supp}
\DeclareMathOperator{\Diff}{Diff}
\DeclareMathOperator{\Evol}{Evol}
\DeclareMathOperator{\Fl}{Fl}
\DeclareMathOperator{\id}{id}

\newcommand{\Frechet}{Fr\'{e}chet}

\newcommand{\LB}[1][\cdot \hspace{1pt} , \cdot]{\left[\hspace{1pt} #1 \hspace{1pt} \right]}

\newcommand{\op}{\mathrm{op}}

\begin{document}

\title{A differentiable monoid of smooth maps on Lie groupoids} \author{Habib Amiri\footnote{University of Zanjan, Iran
\href{mailto:h.amiri@znu.ac.ir}{h.amiri@znu.ac.ir}}\ \ ~~and Alexander
Schmeding\footnote{TU Berlin, Germany
\href{mailto:schmeding@tu-berlin.de}{schmeding@tu-berlin.de}
}%
}
{\let\newpage\relax\maketitle}

\begin{abstract}
In this article we investigate a monoid of smooth mappings on the space of arrows of a Lie groupoid and its group of units.
The group of units turns out to be an infinite-dimensional Lie group which is regular in the sense of Milnor.
Furthermore, this group is closely connected to the group of bisections and the geometry of the Lie groupoid.  
Under suitable conditions, i.e.\ if the source map of the Lie groupoid is proper, one also obtains a differentiable structure on the monoid and can identify the bisection group as a Lie subgroup of its group of units.
Finally, relations between the (sub-)groupoids associated to the underlying Lie groupoid and subgroups of the monoid are obtained.

The key tool driving the investigation is a generalisation of a result by A.\ Stacey. 
In the present article, we establish this so called Stacey-Roberts Lemma. It asserts that pushforwards of submersions are submersions between the infinite-dimensional manifolds of mappings. The Stacey-Roberts Lemma is of independent interest as it provides tools to study submanifolds of and geometry on manifolds of mappings. 
\end{abstract}

\medskip


\medskip

\newpage
\tableofcontents

\section*{Introduction and statement of results} \addcontentsline{toc}{section}{Introduction and statement of results}
In the present article we investigate a monoid $\SG$ of smooth mappings on the arrow space of a Lie groupoid $\cG$.
Namely, for a Lie groupoid $\cG = (G \toto M)$ with source map $\alpha$ and target map $\beta$ we define the monoid
\begin{displaymath}
 \SG \coloneq \{ f \in C^\infty (G,G) \mid \beta \circ f = \alpha \} \quad \text{ with product } (f\star g) (x) \coloneq f(x)g\big(xf(x)\big)
\end{displaymath}
Our motivation to investigate this monoid is twofold.
On one hand, $\SG$ carries over the investigation of a similar monoid for continuous maps of a topological monoid in \cite{Amiri2017} to the smooth category.
While in the continuous setting one always obtains a topological monoid, the situation is more complicated in the smooth category.
This is due to the fact that on non-compact spaces the $C^\infty$-compact open topology has to be replaced by a Whitney type topology (cf.\ \cite{michor1980,HS2016} for surveys).
On the other hand, it has been noted already in \cite{Amiri2017} that $\SG$ is closely connected to the group of bisections $\Bis (\cG)$ of the underlying Lie groupoid $\cG$.
Roughly speaking a bisection gives rise to an element in the unit group of $\SG$ by associating to it its right-translation (we recall these constructions in Section \ref{sect: BIS}).
This points towards the geometric significance of the monoid of $\SG$ and its group of units as bisections are closely connected to the geometry of the underlying Lie groupoid (cf.\ \cite{MR3351079,MR3573833,MR3569066}, see also \cite{Amiri2017}).
Viewing the monoid $\SG$ and its group of units as an extension of the bisections, one is prompted to the question, how much of the structure and geometry of the Lie groupoid $\cG$ can be recovered from these monoids and groups of mappings.

In the present paper we investigate the monoid $\SG$ as a subspace of the space of smooth functions with the fine very strong topology (a Whitney type topology, called $\cF\cD$-topology in \cite{michor1980} which is recalled in Appendix \ref{app: topo}).
It turns out that if the underlying Lie groupoid is $\alpha$-proper, i.e.\ the source map $\alpha$ is a proper map, $\SG$ becomes a topological monoid with respect to this topology.

Throughout the article it will be necessary to work with smooth maps on infinite-dimensional spaces and manifolds beyond the Banach setting, The calculus used here is the so called Bastiani's calculus \cite{bastiani} (often also called Keller's $C^r$-theory~\cite{keller}).
We refer to \cite{Milnor84,BGN04,hg2002a,neeb2006} for streamlined expositions, but have included a brief recollection in Appendix \ref{Appendix: MFD}.
The locally convex setting has the advantage that it is compatible with the underlying topological framework. In particular, smooth maps and differentials of those are automatically continuous.
Thus we directly prove that $\SG$ becomes a submanifold of the infinite-dimensional manifold of smooth mappings which turns $\SG$ into a differentiable monoid if $\cG$ is $\alpha$-proper.
As an indispensable ingredient of the construction, we establish a generalisation of a result by A.\ Stacey. This Stacey-Roberts Lemma enables our investigation and is of independent interest.

Building on our investigation of the monoid $\SG$ we turn to its group of units $\SG(\alpha)$.
It turns out that  the differentiable strucure induced on $\SG(\alpha)$ by $\SG$ turns this group into an infinite-dimensional Lie group (even if $\cG$ is not $\alpha$-proper!).
We investigate the Lie theoretic properties of $\SG (\alpha)$ in Section \ref{sect: unit}.
Then in Section \ref{sect: topsub} we study certain (topological) subgroups of $\SG(\alpha)$, e.g.\ $\Bis  (\cG)$, which are motivated by the geometric interplay of $\cG$ and $\SG(\alpha)$.
Finally, we investigate the connection of these subgroups and certain subgroupoids of $\cG$ in Section \ref{sect: subcon}.
\medskip

We will now go into some more detail and explain the results in this paper.
First recall that the space of smooth mappings $C^\infty (G,G)$ for a finite-dimensional manifold can be made an infinite-dimensional manifold with respect to the fine very strong topology (see Appendix \ref{app: topo} for more details).
To turn $\SG$ and $\SG(\alpha)$ into submanifolds of this infinite-dimensional manifold, we need an auxiliary result:
\smallskip

\textbf{Stacey-Roberts Lemma} \emph{Let $X,Y$ be paracompact finite-dimensional manifolds and $S$ a finite-dimensional manifold, possibly with boundary or corners. Consider  $\theta \colon X \rightarrow Y$ a smooth and surjective submersion. Then the pushforward }
\begin{displaymath}
 \theta_* \colon C^\infty (S,X) \rightarrow C^\infty (S,Y) , \quad f \mapsto \theta \circ f
\end{displaymath}
\emph{is a smooth submersion.}\footnote{One has to be careful with the notion of submersion between infinite-dimensional manifolds, as there is no inverse function theorem at our disposal. For our purposes a submersion is a map which admits submersion charts (this is made precise in Definition \ref{defn: subm} below).}
\smallskip

This result is of independent interest and has been used in recent investigations of Lie groupoids of mappings \cite{1602.07973v2} and shape analysis on homogeneous spaces \cite{CES17}
Before we continue note that the Stacey-Roberts Lemma generalises an earlier result by A.\ Stacey (see \cite[Corollary 5.2]{1301.5493v1}).
In ibid.\ Stacey deals with a setting of generalised infinite-dimensional calculus and states that the pushforward of a submersion is a "smooth regular map" if the source manifold (i.e.\ $S$ for the space $C^\infty (S,X)$) is compact.
From this it is hard to discern the Stacey-Roberts Lemma even if $S$ is compact. It was first pointed out to the authors by D.M.\ Roberts that Stacey's result yields a submersion in the sense of Definition \ref{defn: subm} (see \cite[Theorem 5.1]{1602.07973v2}).
Note that the statement of the Stacey-Roberts Lemma is deceptively simple, as the proof is surprisingly hard (cf.\ Appendix \ref{app: detailedproofs}). To our knowledge a fully detailed proof (building on Stacey's ideas) has up to this point not been recorded in the literature (this is true even for the compact case).

Using the Stacey-Roberts Lemma, one obtains a submanifold structure on $\SG$ and one can ask whether the induced structures turn $\SG$ into a differentiable monoid.
\smallskip

\textbf{Theorem A} \emph{If the Lie groupoid $\cG$ is $\alpha$-proper, i.e.\ the source map is a proper map, then $\SG$ consists only of proper mappings.
Further, the manifold structure induced by $C^\infty (G,G)$ turns $(\SG, \star)$ into a differentiable monoid.}
\smallskip

Note that this in particular implies that $(\SG,\star)$ is a topological monoid with respect to the topology induced by the fine very strong topology.
Due to the first part of Theorem A, $\SG$ contains only proper mappings if $\cG$ is $\alpha$-proper and our proof uses that the joint composition map $\text{Comp} (f,g) \coloneq f\circ g$ is continuous (or differentiable) only if we restrict $g$ to proper maps.

One can prove (see \cite{Amiri2017}) that the monoid $(\SG,\star)$ is isomorphic to a submonoid of $\big(C^\infty (G,G), \circ_\op\big)$ (where  \textquotedblleft$\circ_\op$\textquotedblright\, is composition in opposite order) by means of 
$$R \colon \SG \rightarrow C^\infty (G,G) , \quad f \mapsto \big(x \mapsto x f(x)\big).$$
Note that $R$ makes sense since $x$ and $f(x)$ are composable in $\cG$ by definition of $\SG$.
This monoid monomorphism enables the treatment of the group of units of $\SG$ as an infinite-dimensional Lie group.
To this end, note that the group of units of $\SG$ coincides with the group
\begin{align*}
 \big(\SG (\alpha ) \coloneq \big\{ f \in \SG \mid R(f) \in \Diff (G)\} , \star\big).
\end{align*}
By virtue of the monoid isomorphism, we can identify $\SG(\alpha)$ with the subgroup $R_{\SG (\alpha)} \coloneq \{\varphi \in \Diff (G) \mid \beta \circ \varphi = \beta\}$ of $\Diff (G)$.
Another application of the Stacey-Roberts Lemma shows that $R_{\SG (\alpha)}$ is indeed a closed Lie subgroup of $\Diff (G)$.
This enables us to establish the following result.
\smallskip

\textbf{Theorem B} \emph{Let $\cG$ be a Lie groupoid, then the group of units $\SG(\alpha)$ of the monoid $(\SG, \star)$ is an infinite-dimensional Lie group with respect to the submanifold structure induced by $C^\infty (G,G)$.
Its Lie algebra is isomorphic to the Lie subalgebra}
$$\mathfrak{g}_\beta \coloneq \{X \in \mathfrak{X}_c (G) \mid T\beta \circ X (g) = 0_{\beta (g)} \forall g \in G\}$$
\emph{of the Lie algebra $\big(\mathfrak{X}_c (G), \LB\big)$ of compactly supported vector fields with the usual bracket of vector fields.}\smallskip

We then study Lie theoretic properties of the Lie group $\big(\SG (\alpha), \star\big)$. To understand these results, recall the notion of regularity (in the sense of Milnor) for Lie groups.

Let $H$ be a Lie group modelled on a locally convex space, with identity element $\one$, and $r\in \N_0\cup\{\infty\}$. We use the tangent map of the
 right translation $\rho_h\colon H\to H$, $x\mapsto xh$ by $h\in H$ to define $v.h\coloneq T_{\one} \rho_h(v) \in T_h H$ for $v\in T_{\one} (H) =: \Lf (H)$.
 Following \cite{hg2015}, $H$ is \emph{$C^r$-semiregular} if for each $C^r$-curve $\gamma\colon [0,1]\rightarrow \Lf(H)$ the initial value problem
 \begin{equation}\label{eq: regular}
  \begin{cases}
  \eta'(t)&= \gamma(t).\eta(t)\\ \eta(0) &= \one
  \end{cases}
 \end{equation}
 has a (necessarily unique) $C^{r+1}$-solution $\Evol (\gamma)\coloneq\eta\colon [0,1]\rightarrow H$. If in addition
 \begin{displaymath}
  \mathrm{evol} \colon C^r\big([0,1],\Lf(H)\big)\rightarrow H,\quad \gamma\mapsto \Evol(\gamma)(1)
 \end{displaymath}
 is smooth, we call $H$ a $C^r$-regular Lie group. If $H$ is $C^r$-regular and $r\leq s$, then $H$ is also
 $C^s$-regular. A $C^\infty$-regular Lie group $H$ is called \emph{regular}
 \emph{(in the sense of Milnor}) -- a property first defined in
 \cite{Milnor84}. Every finite dimensional Lie group is $C^0$-regular (cf.\ \cite{neeb2006}).
 Several important results in infinite-dimensional Lie theory are only available for regular Lie groups (see \cite{Milnor84,neeb2006,hg2015}, cf.\ also \cite{conv1997} and the references therein).
Our results subsume the following (cf.\ Corollary \ref{cor: SGA}):
 \smallskip

 \textbf{Theorem C} \emph{The Lie group $\big(\SG (\alpha), \star\big)$ is $C^1$-regular.}
 \smallskip

\noindent
Under mild topological assumptions on the manifold of arrows, one can even sharpen Theorem C to obtain $C^0$-regularity of the Lie group $\big(\SG (\alpha), \star\big)$. This seemingly minor improvement is quite important as it enables one to establish Lie theoretic and geometric tools for the infinite-dimensional group. For example, \cite{Hanusch18}
showed that the strong Trotter formula holds for $C^0$ regular Lie groups. It is currently unknown whether similar results can be obtained for $C^1$-regular Lie groups.\smallskip

 We then study (topological) subgroups of $\SG (\alpha)$ which arise from the action of $\SG$
 \begin{displaymath}
   x.f=R(f)(x) = xf(x) , \quad x \in G, f \in \SG.
 \end{displaymath}
(see e.g.\ \cite[Proposition 6]{Amiri2017}) and the interaction of the geometry of $\cG$ with $\SG$.
For example we investigate subgroups of $\SG(\alpha)$ which arise as setwise stabiliser $F_H\big(\SG(\alpha)\big)$ of a subgroupoid $H\subseteq G$ of $\cG$

Furthermore, the connection between the groups $\Bis (\cG)$ of bisections of the Lie groupoid $\cG$ and $\SG (\alpha)$ is studied.
A bisection $\sigma$ of $\cG = (G\toto M)$ is a smooth map $\sigma\colon  M\rightarrow G$ which satisfies $\beta \circ \sigma = \id_M$ and $\alpha \circ \sigma \in \Diff (M)$.
The group $\Bis (\cG)$ plays an important role in the study of Lie groupoids (cf.\ \cite{MR3573833,ZHUO}) and it is well known (see e.g.\ \cite{Mackenzie05}) that bisections give rise to right-translations of the groupoid $\cG$.
This mechanism yields a group monomorphism $\Psi \colon \Bis (\cG) \rightarrow \SG(\alpha)$ which identifies $\Bis (\cG)$ with a subgroup of $\SG(\alpha)$.
According to \cite[Theorem B]{HS2016}, $\Bis(\cG)$ with the fine very strong topology forms a topological group.
In this case we prove that $\Psi$ becomes a morphism of topological groups if $\cG$ is $\alpha$-proper.
Morover, if $\Bis (\cG)$ carries a suitable Lie group structure, e.g.\ if $M$ is compact,\footnote{Following the sketch provided in \cite[Theorem D]{HS2016}, $\Bis (\cG)$ will be a Lie group even if $M$ is not compact. However, to our knowledge no full proof of this fact exists in the literature.} $\Psi$ is a morphism of Lie groups.

Finally, the last section of the present paper investigates connections between subgroups of $\SG(\alpha)$ and Lie subgroupoids of $\cG$.
The idea here basically stems from the similarity between the structure of groups $\Bis(\cG)$ and $\SG(\alpha)$.
An interesting question is, whether the underlying Lie groupoid $\cG$ is determined (as a set or a manifold) by its group of bisections. For example this question was investigated in \cite{MR3573833} (see also \cite{ZHUO}).
Our aim is to consider related question for $\SG(\alpha)$. Further, we obtain some interesting characterisations of elements in $\SG(\alpha)$ by studying the graphs of its elements in $G\times G$.
Among the main results established in Section \ref{sect: subcon} we prove that
\begin{itemize}
 \item diffeomorphic subgroupoids of $\cG$ give rise to topological isomorphic setwise stabiliser subgroups of $\SG(\alpha)$,
 \item every subgroup of $\SG (\alpha)$ gives rise to a subgroupoid of the groupoid of composable elements $\Gtwo$.
\end{itemize}

\section{Notation and preliminaries}
\label{sec:locally_convex_lie_groupoids_and_lie_groups}

Before we begin, let us briefly recall some conventions that we are using in this paper.

\begin{setup}[Conventions]\label{1}
 We write $ \mathbb{N} := \lbrace 1,2,\dots \rbrace $ and $ \mathbb{N}_0 := \lbrace 0,1,\dots \rbrace $.
 By $G,M, X, Y$ we denote paracompact finite-dimensional manifolds over the real numbers $\mathbb{R}$. We assume that all topological spaces in this article are Hausdorff.
\end{setup}

\begin{setup}[Spaces of smooth mappings]
 We write $C^\infty_{fS} (X,Y)$ for the space of smooth mappings endowed with the \emph{fine very strong topology}.
 Details on this fine very strong topology and the coarser very strong topology are recalled in Appendix \ref{app: topo}.
 Subsets of $C^\infty_{fS} (X,Y)$ will always be endowed with the induced subspace topology. Note that the set of all smooth diffeomorphisms $\Diff (X)$ is an open subset of $C^\infty_{fS} (X,X)$.

 As all spaces of smooth mappings in the present paper will be endowed with this topology, whence we usually suppress the subscript.
 As we recall in Appendix \ref{app: topo} this topology turns $C^\infty (X,Y)$ into an infinite-dimensional manifold (cf.\ Appendix \ref{Appendix: MFD} for the basics of calculus used). This manifold structure turns $\Diff (X)$ into an infinite-dimensional Lie group with respect to the composition of mappings.
\end{setup}

We refer to \cite{Mackenzie05} for an introduction to (finite-dimensional) Lie groupoids. The notation for Lie groupoids and their structural maps also
follows \cite{Mackenzie05}.

\begin{setup}\label{2}
Let $\cG = (G \toto M)$ be a finite-dimensional Lie groupoid over $M$ with source projection
$\alpha \colon G \rightarrow M$ and target projection
$\beta \colon G \rightarrow M$. Note that $\alpha$ and $\beta$ are submersions, whence the set of composable arrows $$G^{(2)} = \big\{ (x,y) \in G\times G \mid \alpha (x) = \beta (y) \big\} \subseteq G \times G$$ is a submanifold.
We denote by $m \colon G^{(2)} \rightarrow G$ the partial multiplication and let $\iota \colon G \rightarrow G$ be the inversion.
In the following we will always identify $M$ with an embedded submanifold of $G$ via the unit map $1 \colon M \rightarrow G$ (in the following we suppress the unit map in the notation without further notice).
\end{setup}

\begin{rem}\label{G^2} One obtains a groupoid $\Gtwo$ by defining:
$$\beta^2(x,y)=(x,y)(xy,y^{-1})=(x,\beta(y))=\big(x,\alpha(x)\big), \text{ target map}$$
$$\alpha^2(x,y)=(xy,y^{-1})(x,y)=(xy,y^{-1}y)=\big(xy,\alpha(xy)\big) \text{ source map}.$$
The object space is then $(\Gtwo)^0 = \big\{ \big(x,\alpha (x)\big) \mid x \in G\big\}$ and the set of composable pairs, $(\Gtwo)^{(2)}$, is given by $\big\{\big((x,y),(z,w)\big):\ z=xy\big\}$ with partial multiplication $(x,y)(xy,w)=(x,yw)$
and inverse $(x,y)^{-1}=(xy,y^{-1})$.

In general, the groupoid $\Gtwo$ will not be a Lie groupoid (as for example $\alpha^2$ and $\beta^2$ will fail to be submersions).
Note however, that $\Gtwo$ becomes a topological groupoid with respect to the subspace topology induced by $G\times G \times G \times G$.
\end{rem}

\section{A monoid of smooth maps of a Lie groupoid}

We are interested in a monoid which was first studied in \cite{Amiri2017} in the context of topological groupoids.
Here we study a related monoid, replacing continuous mappings with smooth ones and topological groupoids by Lie groupoids.

\begin{defn}\label{defn: SG}
Let $\cG = (G \toto M)$ be a Lie groupoid, then we define the set
$$\SG \coloneq \big\{f \in C^\infty (G,G) \mid \beta \circ f = \alpha \big\},$$
which becomes a monoid with respect to the binary operation
$$(f \star g) (x) = f(x)g\big(xf(x)\big), \quad f,g \in S_\cG , \ x \in G.$$
The source map $\alpha$ is the unit element of $S_\cG$.\footnote{Here we identify $M \subseteq G$ via the unit map, i.e.\ without this identification the unit is $1 \circ \alpha$! }
\end{defn}

\begin{rem}
Obviously, one can define another monoid $S_\cG'$ of smooth self maps by switching the roles of $\alpha$ and $\beta$ in the definition.
However, as shown in \cite[Proposition 1]{Amiri2017} this leads to an isomorphic monoid, whence it suffices to consider only $\SG$.
\end{rem}

Instead of directly investigating the $\SG$ as a topological monoid, we establish a differentiable structure which, under suitable conditions, turns $(\SG, \star)$ into a differentiable (infinite-dimensional) monoid. As smooth maps in our setting are continuous\footnote{Since we are working beyond Banach manifolds (even beyond the \Frechet setting) this property is not automatic and depends on the generalisation of calculus used. See e.g.\ \cite{conv1997} for a calculus with smooth discontinuous maps.}, this will in particular turn $(\SG ,\star)$ into a topological monoid (it turns out that the conditions needed to turn $\SG$ into a topological monoid are strong enough to turn it into a differentiable monoid). Before continuing, we urge the reader to recall from Appendix \ref{Appendix: MFD} the basics of infinite-dimensional calculus used.

\subsection*{Interlude: The Stacey-Roberts Lemma}\addcontentsline{toc}{subsection}{Interlude: The Stacey-Roberts Lemma}
To obtain a manifold structure on $\SG$, note that it is the preimage of the singleton $\{\alpha\} \subseteq C^\infty (G,M)$ under the pushforward $\beta_* \colon C^\infty (G,G) \rightarrow C^\infty (G,M), f \mapsto \beta \circ f$.
Due to a theorem of Stacey \cite[Corollary 5.2]{1301.5493v1}, the pushforward of a surjective submersion between finite-dimensional manifolds is a submersion if the source manifold $G$ is compact. To extend this to non-compact source manifolds, note first that by \cite{Glofun} a submersion between infinite-dimensional manifolds should be the following.

\begin{defn}\label{defn: subm}
	Let $X,Y$ be possibly infinite-dimensdional manifolds and  $q \colon X\rightarrow Y$ be smooth. Then we call $q$ \emph{submersion}, if for each $x \in X$, there exists a chart $\Psi \colon U_\Psi \rightarrow V_\Psi \subseteq E$ of $X$ with $x \in U_\Psi$ and a chart $\psi \colon U_{\psi} \rightarrow V_\psi\subseteq F$ of $Y$ with $q (U_\Psi) \subseteq U_\psi$ and $\psi \circ q \circ \Psi^{-1} = \pi|_{V_\Psi}$ for a continuous linear map $\pi \colon E \rightarrow F$ which has a continuous linear right inverse $\sigma \colon F\rightarrow E$, (i.e.\ $\pi \circ \sigma = \text{id}_F$).
\end{defn}

Due to the nature of charts of manifolds of mappings (Appendix \ref{app: topo}), the proof of the following is highly non trivial, whence several auxiliary results are collected in Appendix \ref{app: detailedproofs}.

\begin{lem}[Stacey-Roberts]\label{staceyroberts}
	Let $M,X,Y$ be finite-dimensional, paracompact manifolds and $\theta \colon X \rightarrow Y$ be a smooth surjective submersion.
	Then the pushforward $\theta_* \colon C^\infty (M,X) \rightarrow C^\infty (M,Y)$ is a smooth submersion in the sense of Definition \ref{defn: subm}.
\end{lem}

\begin{proof}
	We already know by \ref{setup: pf} that $\theta_*$ is smooth, whence we only have to construct submersion charts for $\theta_*$.
	
	To this end denote by $\mathcal{V} \subseteq TX$ the vertical subbundle given fibre-wise by $\mathrm{Ker}\, T_p \theta$.
	Lemma \ref{lem: adaptedlocadds} allows us to choose a smooth horizontal distribution $\mathcal{H} \subseteq TX$ (i.e.\ a smooth subbundle such that $T X = \mathcal{V} \oplus \mathcal{H}$) and local additions $\eta_X$ on $X$ and $\eta_Y$ on $Y$ such that the following diagram commutes:
	\begin{equation}\label{eq: diagsrlemma} \begin{aligned}
	\begin{xy}
	\xymatrix{
		TX =  \mathcal{V} \oplus\mathcal{H} \ar[d]^{0 \oplus T\theta|_{\mathcal{H}}} &  \ar[l]_-{\supseteq}\Omega_X \ar[r]^{\eta_X}  & X \ar[d]^{\theta}  \\
		TY & \ar[l]_-{\supseteq} \Omega_Y \ar[r]^{\eta_Y}  &   Y
	}
	\end{xy}\end{aligned}
	\end{equation}
	Now fix $f \in C^\infty (M, X)$ and construct submersion charts for $f$ as follows.
	By \ref{setup: smoothstruct} the manifold structures of $C^\infty (M,X)$ and $C^\infty (M,Y)$ do not depend on the choice of local addition, whence we can use $\eta_X$ and $\eta_Y$ without loss of generality.
	We use the canonical charts $\varphi_f \colon U_f \rightarrow \mathcal{D}_f (M, TX) \cong \Gamma_c (f^*TX)$ and $\varphi_{\theta_*(f)} \colon U_{\theta_*(f)} \rightarrow \mathcal{D}_{\theta_*(f)} (M, TY) \cong \Gamma_c (\theta_* (f)^* TY)$.
	and  identify (cf.\ \cite[Proof of Proposition 4.8 and Remark 4.11 2.]{michor1980})
	\begin{equation}\label{eq: ident}
	\Gamma_c (f^*TX) = \Gamma_c (f^*(\mathcal{V} \oplus \mathcal{H})) \cong  \Gamma_c (f^* \mathcal{V}) \oplus \Gamma_c (f^* \mathcal{H}).
	\end{equation}
	Further, by the universal property of the pullback bundle $T\theta$ induces a bundle morphism $\Theta \colon f^*TX \rightarrow \theta_*(f)^* TY$ over the identity of $M$.
	Note that on the subbundle $f^* \mathcal{H}$ the morphism $\Theta$ restricts to a bundle isomorphism $B \colon f^* \mathcal{H} \rightarrow \theta_* (f)^* TY$ (as $B$ is a bundle morphism over the identity of $M$ which is fibre-wise an isomorphism).
	We denote by $I_f \colon \theta_* (f)^*TY \rightarrow f^* TX$ the bundle map obtained from composing the inverse $B^{-1}$ with the canonical inclusion $i^f_\mathcal{H}\colon f^*\mathcal{H} \rightarrow f^*TX$ (of smooth bundles).
	Further, by a slight abuse of notation we will also denote by $(I_f)_*$ and $(T\theta)_*$ pushforwards of the bundle mappings on the spaces of sections $ \Gamma_c (f^*TX)$ and $\Gamma_c (f^*TY)$.
	Note that these pushforwards are smooth and linear maps (as the vector space structure is given by the fibre-wise operations).
	
	Recall from \ref{setup: smoothstruct} that the the inverses $\varphi_f^{-1}$ and $\varphi_{\theta_* (f)}^{-1}$ are given by postcomposition with the local addition.
	Hence, we can combine \eqref{eq: diagsrlemma}, \eqref{eq: ident} and the definition of $I_f$ to obtain a commutative diagram (where restrictions to open sets are suppressed):
	\begin{equation}\label{eq: diagfinal} \begin{aligned}
	\begin{xy}
	\xymatrix{
		\Gamma_c \big(f^*(\mathcal{H} \cap \Omega_X)\big) \ar[r]^-{(i^f_\mathcal{H})_*} & \Gamma_c \big(f^*(\Omega_X)\big)  \ar[d]^{(T\theta)_*}\ar[r]^{\varphi_f^{-1}} &  C^\infty (M,X) \ar[d]^{\theta_*}  \\
		\Gamma_c \big(\theta_* (f)^* \Omega_Y\big)  \ar@2{-}[r] \ar[u]^{(B^{-1}_*)}&   \Gamma_c \big(\theta_* (f)^* \Omega_Y\big) \ar@<5pt>[u]^{(I_f)_*} \ar[r]^{\varphi_{\theta_* (f)}^{-1}}             &   C^\infty (M, Y)
	}
	\end{xy}\end{aligned}
	\end{equation}
	We conclude that on the canonical charts for $f$, the map $\theta$ satisfies
	$$\varphi_{\theta_* (f)} \circ \theta_* \circ \varphi_f^{-1} = (T\theta)_*|_{\Omega_X}$$
	for the continuous linear map $(T \theta)_*$.
	Vice versa, we deduce from \eqref{eq: diagfinal} that the map $(I_f)_*$ is the continuous linear right inverse of $(T\theta)_*$ turning the pair $(\varphi_f, \varphi_{\theta_* (f)})$ into submersion charts.
\end{proof}

\begin{rem}
	\begin{enumerate}
		\item  Note that the source manifold $M$ in the statement of Proposition \ref{staceyroberts} played no direct r\^{o}le in the proof.
		In particular, Lemma \ref{staceyroberts} stays valid if we replace $M$ by a manifold with boundary or with corners (e.g.\ a compact interval).
		We refer to \cite[Chapter 10]{michor1980} for the description of the manifold structure on $C^\infty (M,N)$ if $M$ is a manifold with corners.
		\item In general the pushforward $\theta_*$ of a surjective submersion $\theta$ will not be surjective. For example, consider a surjective submersion $\theta \colon X \rightarrow Y$ which does not define a fibration and let $Y$ be connected. Then curves to $Y$ will in general not lift globally (since $\theta$ does not admit a complete Ehresmann connection, see \cite{MR3542951}).
	\end{enumerate}
\end{rem}

The Stacey-Roberts Lemma now turns $\SG$ into a submanifold of $C^\infty_{fS} (G,G)$.

\begin{lem}\label{SG: sbmfd}
	Let $\cG = (G \toto M)$ be a Lie groupoid. Then $\SG$ is a closed and split submanifold of $C^\infty_{fS} (G,G)$.
\end{lem}

\begin{proof}
	Note that $\SG$ is the preimage of the singleton $\{ \alpha \} \subseteq C^\infty (G,M)$ under the submersion $\beta_* \colon C^\infty (G,G) \rightarrow C^\infty (G,M)$.
	Thus the statement follows from the usual results on preimages of submanifolds under submersions (see e.g.\ \cite[Theorem C]{Glofun} for the statement in our infinite-dimensional setting).
\end{proof}

The submanifold structure enables us to study differentiability conditions of the monoid operation and we begin now with a few preparations.

\begin{lem}\label{con3}
Let $\cG$ be a Lie groupoid then the following mappings are smooth:
\begin{enumerate}
 \item $\Gamma \colon \SG\rightarrow C^{\infty}(G,\Gtwo)$ with $\Gamma(f)(x)\coloneq \big(x,f(x)\big),$ for $x\in G$.
 \item $R\colon\SG\rightarrow C^{\infty}(G,G)$ by $R(f)(x)\coloneq R_f(x) \coloneq xf(x), x \in G$. \label{con3: 2}
\end{enumerate}
\end{lem}
\begin{proof}
\begin{enumerate}
 \item By definition of $\SG$ the map $\Gamma$ makes sense.
 We let now $\Delta \colon G \rightarrow G \times G$ be the diagonal map. By \cite[Proposition 10.5]{michor1980}, this map induces a diffeomorphism
 $$C^\infty (G,G) \times C^\infty (G,G) \rightarrow C^\infty (G,G \times G),\quad (f,g) \mapsto (f,g)\circ \Delta.$$
 Restricting the diffeomorphism to the closed submanifold $C^\infty (G,G) \times \SG \subseteq C^\infty (G,G) \times C^\infty (G,G)$, it factors to a map $\delta \colon C^\infty (G,G) \times \SG \rightarrow C^\infty (G,\Gtwo)$. Since $G^{(2)} \subseteq G\times G$ is an embedded submanifold, \cite[Proposition 10.8]{michor1980} implies that $C^\infty (G,\Gtwo)$ is a closed and split submanifold of $C^\infty (G,G\times G)$, whence $\delta$ is smooth.
 Then $\Gamma (f) = \delta (\id_G , f)$ is smooth as $\SG \rightarrow C^\infty (G,G) \times \SG , f \mapsto (\id_G,f)$ is smooth.
 \item Since $R=m_* \circ \Gamma$ where $m_*$ is the pushforward of the (smooth) multiplication $m\colon \Gtwo\rightarrow G$ of $\cG$. Thus \ref{setup: pf} shows that $R$ is smooth.\qedhere
\end{enumerate}
\end{proof}

\begin{setup}\label{setup: RISO}
 Let $\cG=(G\toto M)$ be a Lie groupoid then we define
 $$R_{\SG} \coloneq R(\SG) = \big\{ R_f \mid f\in\SG \big\} \subseteq C^\infty(G,G).$$
As \cite[Lemma 1]{Amiri2017} shows, $R$ satisfies $R(f\star g) = R(g) \circ R(f)$ for all $f,g \in \SG$ and is injective. In the following, we denote by \textquotedblleft $\circ_\op$\textquotedblright\, composition in opposite order, i.e.\ $f\circ g = g\circ_\op f$.
Then $(R_{\SG}, \circ_\op)$ is a submonoid of $(C^\infty (G,G), \circ_\op)$ and with respect to this structure $R$ induces a monoid monomorphism.
\end{setup}

\begin{rem}\label{rem: gpbd}
Note that if $\cG$ is a group bundle then $\SG=R_{\SG}$ and if $\SG\cap R_{\SG}\neq\emptyset$ then $G$ is a group bundle.
Therefore if $\cG$ is not a group bundle, then $\SG$ and $R_{\SG}$ are disjoint subsets of $C^{\infty}(G,G)$.
\end{rem}

\begin{setup}[The isomorphism $R \colon \SG \rightarrow R_{\SG}$] \label{setup: mon:iso}
By \ref{setup: RISO} the monoids $(R_{\SG}, \circ_{\op})$ and $(\SG, \star)$ are isomorphic, where the inverse of the isomorphism $R$ is given by $R^{-1} (\psi) (x) \coloneq x^{-1} \psi (x)$ for $\psi \in R_{\SG}$ and $x \in G$.
We further note that $R(\alpha) = \id_{G}$ and we have the identity
$$R_{\SG} = \big\{f \in C^\infty (G,G) \mid \beta \circ f = \beta\big\} = \beta_*^{-1} (\{\beta\}). $$
As a consequence of the Stacey-Roberts Lemma \ref{staceyroberts} also $R_{\SG}$ is a closed and split submanifold of $C^\infty (G,G)$.
Now Lemma \ref{con3} implies that $R$ induces a diffeomorphism between the submanifolds $\SG$ and $R_{\SG}$.
To see this note that $R^{-1} = m_* \circ (\iota ,\id_{\SG}) \circ \Delta$ (where the notation is as in Lemma \ref{con3}), whence also $R^{-1}$ is smooth.

Thus the monoids $(\SG,\star)$ and $(R_{\SG},\circ_{\op})$ are isomorphic via the diffeomorphism $R|^{R_{\SG}}$.
In particular, $(R_{\SG},\circ_{\op})$ is a differentiable monoid if and only if $(\SG,\star)$ is a differentiable monoid.
This will enable us to investigate the group of invertible elements of $\SG$ in Section \ref{sect: unit}.
\end{setup}

\subsection*{Differentiable monoids for source-proper groupoids}\addcontentsline{toc}{subsection}{Differentiable monoids for source-proper groupoids}
To establish smoothness of the monoid operation "$\star$", we would like to use smoothness of the composition $\circ \colon C^\infty (Y,Z) \times C^\infty (X,Y) \rightarrow C^\infty (X,Z), (f,g) \mapsto f \circ g$. However, this mapping will in general not be smooth (not even continuous!) if $X$ is non-compact. Following \ref{setup: comp}, we can obtain a smooth map if we restrict to the subset of smooth proper maps $\Prop (X,Y)$. It turns out that the $\SG$ will be contained in this set if the Lie groupoid satisfies the following condition.

\begin{setup}
	The Lie groupoid $\cG = (G \toto M)$ is \emph{$\alpha$-proper} (or \emph{source-proper}) if the source map is a proper map, i.e.\ $\alpha^{-1} (K)$ is compact for each $K \subseteq M$ compact (cf.\ \cite[\S 10.3 Proposition 7]{MR1726779}).
	Note that this entails that $\beta$ is a proper map, since $\beta = \alpha \circ \iota$ and $\iota$ is a diffeomorphism.
\end{setup}

	The concept of $\alpha$-proper groupoids appears for example in the integration of Poisson manifolds of compact type, see \cite{1603.00064v1}.
	
\begin{ex}
	Let $G$ be a Lie group acting smoothly on $M$. We denote by $G \ltimes M$ the corresponding action Lie groupoid.
	Then $G \ltimes M$ is $\alpha$-proper if and only if $G$ is compact.
	Note that $\alpha$-properness is a stronger condition than being a proper Lie groupoid (which in the $G \ltimes M$ example would only force the group action to be proper).
\end{ex}

\begin{ex}
Recall from \cite{MR3638284,MR2012261} that every paracompact, smooth and effective orbifold can be represented by a so called atlas groupoid. To this end, one constructs from an atlas a proper \'{e}tale Lie groupoids (see \cite[Proposition 5.29]{MR2012261}). Following this procedure for a locally finite orbifold atlas, one sees that the atlas groupoid will even be source proper. Note that an atlas which is not locally finite yields a non source proper Lie groupoid. Since all atlas groupoids of a fixed orbifold are Morita equivalent, source properness is not stable under Morita equivalence.  
\end{ex}

\begin{rem}
	In general, it is not enough to require that the $\alpha$-fibres, $\alpha^{-1} (x) , x \in M$ are compact to obtain an $\alpha$-proper groupoid.
	However, it is sufficient to require that the $\alpha$-fibres of $\cG$ are compact and connected, to obtain an $\alpha$-proper groupoid.
	This can be seen as follows: Consider a quotient map $f \colon X \rightarrow Y$ between Hausdorff locally compact spaces such that every component of $f^{-1} (x)$ is compact. By \cite[p. 103]{MR526764} Then $f = l \circ \varphi$ uniquely factors into a proper map $\varphi \colon X \rightarrow M$ (using that closed maps with $f^{-1} (x)$ compact for all $x$ are proper by \cite[\S 10.2 Theorem 1]{MR1726779}) onto a quotient space $M$ and a map $l \colon M \rightarrow Y$.
	If we assume that the fibres $f^{-1} (x)$ are connected then the quotient space $M$ coincides with $Y$ and thus $\varphi$ coincides with $f$.
\end{rem}

\begin{lem}\label{proper}
For an $\alpha$-proper Lie groupoid $\cG$ one has $\SG \subseteq \Prop (G,G)$ and $R(\SG) \subseteq \Prop (G,G)$.
\end{lem}

\begin{proof}
Let us first fix $f \in \SG$ and prove that $f$ is a proper map.
It suffices to prove that $f^{-1} (K)$ is compact for each compact $K \subseteq G$.
Now by definition of $\SG$ we have $f^{-1} (K) \subseteq \alpha^{-1} \big(\beta (K)\big)$.
Hence by $\alpha$-properness, $f^{-1}(K)$ is compact as a closed subset of the compact set $\alpha^{-1} \big(\beta (K)\big)$.

Similarly we proceed for $R(f)$ as \ref{setup: mon:iso} implies that $R_f^{-1}(K)\subseteq \beta^{-1}\big(\beta(K)\big)$. Now $\cG$ is $\alpha$-proper, whence $\beta$ is proper (as $\beta = \alpha \circ \iota$) and therefore $R_f^{-1}(K)$ as a closed subset of a compact set $\beta^{-1}\big(\beta(K)\big)$ which again is compact.
\end{proof}

We can now prove the main theorem of this section which establishes continuity and differentiability of the monoid $\SG$ for $\alpha$-proper Lie groupoids.

\begin{thm}\label{thm: SG:diffgpd}
	Let $\cG = (G \toto M)$ be an $\alpha$-proper Lie groupoid. Then $(\SG , \star)$ with the manifold structure from Lemma \ref{SG: sbmfd} is a differentiable monoid.
\end{thm}

\begin{proof}
	Recall from Definition \ref{defn: SG} that the monoid multiplication can be written as
	\begin{displaymath}
	f \star g  = m_* \circ \Big(f , \mathrm{Comp} \big(g, R(f)\big)\Big) \quad f,g \in \SG ,
	\end{displaymath}
	where $\mathrm{Comp} (f,g) = f\circ g$ is the joint composition map.
	By Lemma \ref{con3} $R$ is a smooth map which maps $\SG$ into the open subset $\Prop (G,G)$ by Lemma \ref{proper} and $\alpha$-properness of $\cG$.
	Since $\mathrm{Comp} \colon C^\infty (G,G) \times \Prop (G,G) \rightarrow C^\infty (G,G)$ is smooth by \ref{setup: comp} the monoid operation $\star$ is smooth as a map into the (closed) submanifold $\SG$ (cf.\ \cite[p.\ 10]{Glofun}).
\end{proof}

\begin{rem}
We established continuity and differentiability of the multiplication in $\SG$ by exploiting the continuity of the joint composition map.
If the Lie groupoid $\cG$ is not $\alpha$-proper, $\SG$ contains mappings which are not proper, e.g.\ $\alpha$, and the joint composition is in general discontinuous outside of $\Prop (G,G)$ (see e.g.\ \cite[Example 2.1]{HS2016}.
Thus the proof of Theorem \ref{thm: SG:diffgpd} breaks down and we do not expect that $\SG$ will in general be a topological monoid.
Note however, that the subgroup of units in $\SG$ (studied in Section \ref{sect: unit} below) will always become a Lie group even if $\cG$ is not $\alpha$-proper.
\end{rem}

\begin{rem}
Analogous results to the ones established in this section can be established for the coarser very strong topology. One can prove that $\SG$ is a topological monoid in the very strong topology if $\cG$ is $\alpha$-proper. As the very strong topology does not turn $C^\infty (G,G)$ into a manifold, there is no differentiable structure on $\SG$.
\end{rem}

\section{The Lie group of invertible elements} \label{sect: unit}

In this section we consider the group of invertible elements of the monoid $(\SG,\star)$.
It will turn out that this group is always a Lie group (endowed with the subspace topology of the (fine) very strong topology).
Let us first derive an alternative description of the group of invertible elements in $\SG$.

\begin{defn}
 For a Lie groupoid $\cG$ define the group
 \begin{displaymath}
  \big(\SG (\alpha) \coloneq \big\{f \in \SG \mid R(f) \in \Diff (G)\big\} , \star\big).
 \end{displaymath}
 Further, we set $R_{\SG (\alpha)} \coloneq R\big(\SG (\alpha)\big)$.
\end{defn}

\begin{setup}
 In \ref{setup: mon:iso} we have seen that $R$ is a monoid isomorphism between $(\SG,\star)$ and
 $$\big(R_{\SG} = \big\{f \in C^\infty (G,G) \mid \beta \circ f = \beta \big\}, \circ\big).$$ As the units in $R_{\SG}$ are clearly given by $R_{\SG (\alpha)} = R_{\SG} \cap \Diff (G)$, we see that $\big(\SG (\alpha), \star\big)$ indeed is the group of units of the monoid $\SG$.

 Now the group $\Diff(G)$ is open in $C^{\infty}_{fS}(G,G)$ and Lemma \ref{con3} \ref{con3: 2} implies that $R$ is smooth, $\SG(\alpha) = R^{-1}\big(\Diff(G)\big)$ is an open submanifold of $\SG$.
\end{setup}

The fact that $R_{\SG(\alpha)}\subseteq \Diff(G)\subseteq \Prop(G,G)$ will enable us to show that for an arbitrary Lie groupoid $\cG$ the group of units in $\SG$ is an infinite-dimensional Lie group. As $\SG (\alpha)$ is isomorphic as a group and a submanifold to $R_{\SG (\alpha)}$ we will first establish a Lie group structure on  $R_{\SG (\alpha)}$. To this end recall that for a finite-dimensional manifold $G$, $\Diff (G)$ is an open subset of $C^\infty_{fS} (G,G)$.
By \cite[Theorem 11.11]{michor1980} this structure turns $(\Diff (G), \circ)$ (or equivalently $(\Diff (G),\circ_\op)$) into an infinite-dimensional Lie group whose Lie algebra is given by the compactly supported vector fields $\mathfrak{X}_c (G)$ whose Lie bracket is the negative of the bracket of vector fields (cf.\ \cite{MR3328452}).

\begin{prop}\label{prop: RSG:Lie}
  Let $\cG = (G \toto M)$ be a Lie groupoid. Then $R_{\SG (\alpha)}$ is a closed Lie subgroup of $(\Diff (G),\circ)$.
\end{prop}

\begin{proof}
	Being a submersion is a local condition, whence $\beta_*|_{\Diff (G)}$ is a submersion as a consequence of the Stacey-Roberts Lemma \ref{staceyroberts}.
	Clearly $R_{\SG(\alpha)} = R_{\SG} \cap \Diff (G) = (\beta_*|_{\Diff (G)})^{-1} (\{\beta \})$ is a closed submanifold of $\Diff (G)$ and this structure coincides with the manifold structure obtained from $\Diff (G)\cap R_{\SG}$.
	Thus $R_{\SG (\alpha)}$ is a closed submanifold of $\Diff (G)$, whence it suffices to check that it is a subgroup of $(\Diff(G), \circ)$:
	Let $f,g \in R_{\SG (\alpha}$, then clearly $\beta \circ (f \circ g) = \beta$. Moreover, for the inverse $f^{-1}$ of $f$ in $\Diff (G)$ we have $ \beta \circ f^{-1} = (\beta \circ f) \circ f^{-1} = \beta$ and thus $f^{-1} \in R_{\SG (\alpha)}$ if and only if $f \in R_{\SG (\alpha)}$. Summing up $R_{\SG (\alpha)}$ is a closed Lie subgroup of $(\Diff (G), \circ)$.\footnote{Note that in contrast to the finite dimensional case, closed subgroups of infinite-dimensional Lie groups need not be Lie groups (cf.\ \cite[Remark IV.3.17]{neeb2006} for an example), whence it was essential to prove that $R_{\SG (\alpha)}$ is a closed submanifold of $\Diff(G)$.}
\end{proof}

To identify the Lie algebra of $R_{\SG (\alpha)}$ (as a subalgebra of $\Lf ((\Diff (G),\circ)) = \mathfrak{X}_c (G)$) we need some preparations.

\begin{setup}
 Consider the subset $T^\beta G = \bigcup_{g \in G} T_g \beta^{-1} \beta (g)$ of $TG$.
 Note that for all $v \in T_g^\beta G$ the definition implies $T \beta(v) = 0_{\beta (g)} \in T_\beta(g) M$, i.e.\ fibre-wise we have $T_g^\beta G =
\mathrm{ker} T g \beta$. Since $\beta$ is a submersion, the same is true for $T \beta$. Computing in submersion charts,
the kernel of $T g \beta$ is a direct summand of the model space of $T G$. Furthermore, the submersion
charts of $T \beta$ yield submanifold charts for $T^\beta G$ whence $T^\beta G$ becomes a split submanifold
of $T G$. Restricting the projection of $T G$, we thus obtain a subbundle $\pi_\beta \colon T^\beta G \rightarrow G$ of the
tangent bundle $T G$.
\end{setup}

\begin{prop}\label{prop: RSG:LA}
 The Lie algebra of $(R_{\SG (\alpha)}, \circ)$ is given by the Lie subalgebra
 $$\mathfrak{g}_\beta \coloneq \big\{X \in \mathfrak{X}_c (G) \mid X(G) \subseteq T^\beta G\big\}$$
 of $\big(\mathfrak{X}_c (G), \LB\big)$, where $\LB$ is the negative of the usual bracket of vector fields
\end{prop}

\begin{rem}
 Note that it is clear a priori that $\mathfrak{g}_\beta$ is a Lie subalgebra of $\big(\mathfrak{X}_c (G), \LB\big)$ since $X \in \mathfrak{g}_\beta$ if and only if $X$ is $\beta$-related to the zero-vector field on $M$.
 (cf.\ \cite[Lemma 3.5.5]{Mackenzie05} for a similar construction in the context of the Lie algebroid of $\cG$). However, this will also follow from our proof below.
\end{rem}

\begin{proof}[Proof of Proposition \ref{prop: RSG:LA}]
	Denote by $C^1_* \big([0,1], \Diff (G)\big)$ the set of all continuously differentiable curves $c \colon [0,1] \rightarrow \Diff (G)$ with $c(0) = \id_G$.
	Since $R_{\SG (\alpha)}$ is a closed Lie subgroup of $\Diff (G)$ by Proposition \ref{prop: RSG:Lie}, its Lie algebra $\Lf (R_{\SG (\alpha)})$ can be computed by \cite[Proposition II.6.3]{neeb2006} as the differential tangent set
	$$\Lf^d (R_{\SG (\alpha)}) \coloneq \big\{ \dot{c} (0) \in \mathfrak{X}_c (G) \mid c \in C^1_* \big([0,1],\Diff (G)\big), c\big([0,1]\big) \subseteq R_{\SG (\alpha)}\big\}.$$
	Let us consider a $C^1$-curve $c \colon [0,1] \rightarrow R_{\SG (\alpha)}$ with $c(0) = \id_G$. By defintion of $R_{\SG (\alpha)}$ we have $\beta_* \circ c (t) = \beta$ for all $t \in [0,1]$ and the right hand side does not depend on $t$.
	Hence, with the notation of \ref{setup: smoothstruct}, we obtain
	$$ \left.\frac{\dd}{\dd t}\right|_{t=0} \beta_* \circ c (t) = (T\beta)_* (\dot{c}(0)) = \mathbf{0} \in \cD_\beta (G,TM).$$
	Here we used the formula $T(\beta_*) = (T\beta)_*$ from \cite[Corollary 10.14]{michor1980} and conclude that the compactly supported vector field $\dot{c}(0)$ takes its values in $T^\beta G$, whence $\Lf^d (R_{\SG (\alpha)}) \subseteq \mathfrak{g}_\beta$.
	
	For the converse let $X \in \mathfrak{g}_\beta$. Recall that the flow map $\Fl_1 \colon \mathfrak{X}_c (G) \rightarrow \Diff_c (G)$ to time $1$ is the Lie group exponential of $\Diff (G)$ (see \cite[4.6]{MR702720}).
	Thus we can exponentiate $X$ to the smooth one-parameter curve $c_X \colon [0,1] \rightarrow \Diff_c (G), t \mapsto \Fl_1 (tX)$ with $c_X (0) = \id_G$ and $\dot{c}_X (0) = X$.
	To see that $c_X$ takes its values in $R_{\SG (\alpha)}$, note that $c_X (0) \in R_{\SG (\alpha)}$ and $R_{\SG (\alpha)}$ is a closed Lie subgroup. Thus it suffices to prove that the derivative of $\beta_* \circ c_X$ vanishes.
	However, arguing again as above we have
	\begin{displaymath}
	 \left.\frac{\dd}{\dd t}\right|_{t=0} \beta_* \circ c_X (t) = (T\beta)_* \dot{c}_X (0) = (T\beta)_* (X) = \mathbf{0} \in \cD_\beta (G,TM)
	\end{displaymath}
        since $X$ takes its values in $T^\beta G$ as $X \in \mathfrak{g}_\beta$.
        Summing up $\Lf^d (R_{\SG (\alpha)}) = \mathfrak{g}_\beta$ and the Lie bracket coincides with the restriction of the bracket on $\mathfrak{X}_c (G)$ to the subspace.
\end{proof}

Our next aim is to establish regularity in the sense of Milnor for the subgroup $R_{\SG (\alpha)}$.
The idea here is again to leverage that $R_{\SG (\alpha)}$ is a closed Lie subgroup of $\Diff (G)$.

\begin{setup}[Regularity of $\Diff (G)$]\label{setup: reg:Diff}
 For a paracompact manifold $G$ it is known (see \cite[Corollary 13.7]{hg2015}) that $\Diff (G)$ is $C^1$-regular (even $C^0$-regular if $G$ is $\sigma$-compact), i.e.\ for all $C^1$-curves $\eta \colon [0,1] \rightarrow \mathfrak{X}_c (G)$ the equation
\begin{equation}\label{eq: reg}
 \begin{cases}
  \gamma'(t) &= \eta(t). \gamma (t) \coloneq T\rho_{\gamma (t)} (\eta (t)),\\
  \gamma (0) &= \id_G
 \end{cases}
\end{equation}
where $T\rho_{\gamma (t)}$ is the tangent map of the right translation $\rho_{\gamma (t)}$ in $\Diff (G)$ has a unique solution $\gamma_\eta$ such that
$\Evol \colon C^\infty \big([0,1],\mathfrak{X}_c (G)\big) \rightarrow C^2 \big([0,1], \Diff (G)\big), \eta\mapsto \gamma_\eta$ is smooth.

In the case of $\Diff (G)$, the solution of \eqref{eq: reg} is the flow
\begin{displaymath}
 \Fl^\eta \colon [0,1] \rightarrow \Diff (G) , \quad t \mapsto \Fl^\eta (t,\cdot)
\end{displaymath}
 of the time dependent vector field $\eta$, i.e.\ for every $x \in G,$ $t \mapsto \Fl^\eta (t)(x)$ is the integral curve of $\eta^\wedge \colon [0,1] \times G \rightarrow TG , (t,y) \mapsto \eta(t)(y)$ starting at $x$.
 Thus for $x \in G$ and every $t_0 \in [0,1]$ we have
      \begin{equation}\label{intcurve}
       \left.\frac{\partial}{\partial t}\right|_{t=t_0} \Fl^\eta (t,x) = \eta \big(t_0, \Fl^{\eta} (t_0,x)\big)
      \end{equation}
 For $G$ compact, a proof of this can be found in \cite[p.\ 1046]{Milnor84}. It is also true in the non-compact case as a special version of \cite[Section 5.4]{MR3328452} (in particular \cite[Proof of Theorem 5.4.11]{MR3328452}) where the case of non-compact orbifolds is treated.\footnote{To our knowledge \cite{MR3328452} is the only source currently in print where a full proof of the regularity of $\Diff(G)$ for the non-compact $G$ in our setting can be found. See however \cite[Theorem III.4.1.]{neeb2006} and \cite[Corollary 13.7]{hg2015} where an unpublished preprint by H.\ Gl\"{o}ckner which paved the way for the treatment in \cite{MR3328452} is referenced. Further, we refer to \cite[4.6]{MR702720} for related considerations.}
\end{setup}

\begin{lem}\label{lem: semireg}
 Let $\eta \colon [0,1] \rightarrow \mathfrak{g}_\beta \subseteq \mathfrak{X}_c (G)$ be a $C^k$-curve, where $k=0$ if $G$ is $\sigma$-compact and $k=1$ if $G$ is only paracompact.
 Then the solution $\gamma_\eta$ of \eqref{eq: reg} factors to a map $\gamma_\eta \colon [0,1] \rightarrow R_{\SG (\alpha)}$.
\end{lem}

 \begin{proof}
  In the following we set $k=0$ if $G$ is $\sigma$-compact and $k=1$ if $G$ is only paracompact.
  Since $\Diff (G)$ is $C^k$-regular the solution $\Evol (\eta) \in C^{k+1} \big([0,1], \Diff (G)\big)$ of \eqref{eq: reg} exists in $\Diff(G)$.
  Now $\Evol (\eta)(0) = \id_G$ and $R_{\SG (\alpha)}$ is the closed subgroup of all elements $g$ in $\Diff (G)$ which satisfy $\beta \circ g = \beta$.
  Thus it suffices to prove that $\beta_* \circ \Evol (\eta)$ is constant in $t$, i.e.\ we will show that its derivative vanishes in $T C^\infty (G,M)$.

  Due to \cite[Lemma 10.15]{michor1980} (the proof in loc.cit.\ for $C^\infty$-curves carries over verbatim to $C^1$-curves) the derivative of $c(t) \coloneq \beta_* \circ \Evol (\eta)$ at $t_0$ vanishes in $T_{c(t_0)} C^\infty (G,M) \cong \cD_{c(t_0)} (G,TM)$ if and only if for each $x \in G$ we have $T_{t_0} \big(\ev_x \circ c(t)\big) = 0$, where $\ev_x \colon C^\infty (G,M) \rightarrow M, f \mapsto f(x)$.
  Using now that $\ev_x \circ \beta_* = \beta \circ \ev_x$, we compute
   \begin{align*}
    T_{t_0} \ev_x \circ c (t) &= T_{t_0} \beta \circ \ev_x \big(\Fl^\eta (t,\cdot)\big) = T\beta \left(\left.\frac{\partial}{\partial t}\right|_{t=t_0} \Fl^\eta (t,x)\right) \stackrel{\eqref{intcurve}}{=} T\beta \Big(\eta \big(t,\Fl^\eta (t,x)\big)\Big) \\
			      &= 0 \quad \quad \text{since } \eta \text{ takes its image in } T^\beta G . \qedhere
   \end{align*}
 \end{proof}

 \begin{prop}\label{prop: RSG:reg}
  The group $(R_{\SG (\alpha)}, \circ)$ is $C^1$-regular and even $C^0$-regular if $G$ is $\sigma$-compact.
 \end{prop}

 \begin{proof}
  Due to Lemma \ref{lem: semireg} we know that $R_{\SG (\alpha)}$ is $C^1$-semiregular ($C^0$-semiregular if $G$ is $\sigma$-compact).
  Since $R_{\SG (\alpha)}$ is a closed Lie subgroup of $\Diff (G)$, the assertion then follows from \cite[Lemma B.5]{MR3573833}.
 \end{proof}

 By \ref{setup: RISO} the groups $(R_{\SG (\alpha)}, \circ_\op)$ and $(\SG (\alpha),\star)$ are isomorphic via $R$ and this isomorphism is a diffeomorphism with respect to the natural manifold structures.
 As $(R_{\SG (\alpha)}, \circ_\op)$ and $(R_{\SG (\alpha)}, \circ)$ are anti-isomorphic as Lie groups (e.g.\ by composing $R$ with the group inversion), we conclude from Propositions \ref{prop: RSG:Lie}, \ref{prop: RSG:LA} and \ref{prop: RSG:reg} the following.

 \begin{cor}\label{cor: SGA}
  Let $\cG$ be a Lie groupoid, then the group $(\SG (\alpha), \star)$ is a $C^1$-regular Lie group whose Lie algebra is anti-isomorphic to the algebra $\mathfrak{g}_\beta$.
  If the manifold $G$ is $\sigma$-compact, then $(\SG(\alpha), \star)$ is even $C^0$-regular.
 \end{cor}

Note that in the $C^0$-regular case the validity of the strong Trotter formula follows from \cite{Hanusch18}. It is currently unknown whether  similar results hold for $C^1$-regular Lie  groups. 

\section{Topological subgroups of the unit group}\label{sect: topsub}

In this section we investigate certain topological subgroups of the Lie group $\SG (\alpha)$.
Albeit these groups will often be closed topological groups, we remark that not every closed subgroup of an infinite-dimensional Lie group is a Lie subgroup (cf.\ \cite[Remark IV.3.17]{neeb2006}).
Hence in general we do not know whether these groups carry a differentiable structure turning them into Lie subgroups

\subsection*{The action of the unit group and its stabilisers} \addcontentsline{toc}{subsection}{The action of the unit group and its stabilisers} \label{ss: act}
We first consider the group action of $\SG (\alpha)$ on $\cG$ and subgroups related to the stabilisation of Lie subgroupoids under this action.

\begin{rem} Let $\cG = (G \toto M)$ be a Lie groupoid.
By \cite[Proposition 6]{Amiri2017} the monoid $\SG$ acts on $G$ from the right by
\begin{displaymath}
 x.f=R_f(x)=xf(x) , \quad x \in G, f \in \SG.
\end{displaymath}
Indeed $x.\alpha=x$ and $x.(f\star g)=(x.f).g$ for $x\in G$ and $f,g\in\SG$. Note that this action is faithful, i.e.\ the only element of $\SG$ acting trivially is the identity element $\alpha$ of $\SG$.
It is easy to check that the right-action of $\SG$ is transitive if and only if $G$ is a group, i.e.\ $\cG = (G \toto \ast)$.
\end{rem}
\begin{defn}
Let $\cG = (G \toto M)$ be a Lie groupoid and $H \subseteq G$.
Define
  \begin{align*}
   I_H(\SG) \coloneq \big\{f\in\SG: f(H)\subseteq H\big\}, \quad I_H\big(\SG(\alpha)\big) \coloneq \big\{f\in\SG(\alpha): f(H)\subseteq H\big\}
  \end{align*} and the \emph{setwise stabiliser}
$F_H\big(\SG(\alpha)\big) \coloneq \big\{f\in\SG(\alpha): R_f(H)=H\big\}$ \emph{of $H$}.
\end{defn}
\begin{rem}
\begin{enumerate}
 \item Using the group homomorphism $R\colon \SG (\alpha) \rightarrow \Diff (G)$, it is easy to check that for a subgroupoid $H$ of $G$,
$F_H\big(\SG(\alpha)\big)$ is a subgroup of $\SG(\alpha)$. In addition, $I_H(\SG)$ and $I_H\big(\SG(\alpha)\big)$ are then two submonoids of $\SG$.
\item Since point evaluations are continuous with respect to the (fine) very strong topology (see Corollary \ref{cor: ev}), if $H$ is closed, then the sets $I_H\big(\SG(\alpha)\big), I_H(\SG)$ and $F_H\big(\SG(\alpha)\big)$ are closed in $\SG$.
\item For $f\in I_H\big(\SG(\alpha)\big)$ and  $y\in H$ there exist $x\in G$ with $R_f(x)=xf(x)=y$. However, since it is not clear that $f(x)\in H$, the inverse $g$ of $f$ in $\SG (\alpha)$ does not necessarily satisfy $g(H) \subseteq H$.
As $H$ may not be invariant under $g$, the monoid $I_H\big(\SG(\alpha)\big)$ may not be closed under the inversion map.
\end{enumerate}
\end{rem}

In the following we show that Lie groupoid isomorphisms over the identity (i.e.\ a Lie groupoid isomorphism $(\psi, \id_M)$ induce an isomorphism of the topological groups $F_H\big(\SG(\alpha)\big)$ and $F_{\psi (H)}\big(\SG(\alpha)\big)$.

\begin{prop}\label{Iso1}
Let $\cG = (G \toto M)$ and $\cG' = (G' \toto M)$ be Lie groupoids and $\psi \colon G \rightarrow G'$ be an isomorphism of Lie groupoid over $\id_M$.
Let $H \subseteq G$ be a wide Lie subgroupoid of $\cG$ and $K\coloneq \psi(H)$, then $F_H\big(\SG(\alpha)\big)$ and $F_K\big(S_{G'}(\alpha')\big)$ are isomorphic as topological groups (with the (fine) very strong topology inherited from $C^{\infty}(G,G)$ and $C^\infty (G',G')$, respectively) via 
\begin{displaymath}
 \Psi \colon F_H\big(\SG(\alpha)\big) \rightarrow F_K\big(S_{G'}(\alpha')\big) ,  f \mapsto \psi \circ f\circ \psi^{-1}.
 \end{displaymath}
\end{prop}
\begin{proof}
Note first that for $f\in F_H\big(\SG(\alpha)\big)$ the map $\Psi(f)=\psi\circ f\circ\psi^{-1} \colon G' \rightarrow G'$ satisfies $R_{\Psi (f)} (x) = x \psi \Big(f\big(\psi^{-1}(x)\big)\Big) = \psi \Big(R_f \big(\psi^{-1} (x)\big)\Big)$.
Since $\psi$ and $R_f$ are diffeomorphisms, $R_{\Psi (f)} \colon G' \rightarrow G'$ is a diffeomorphism, whence $\Psi (f) \in S_{G'} (\alpha')$.
Similarly,
\begin{align*}
R_{\Psi (f)}(K)%\Big\{x\big(\psi\circ f \circ\psi^{-1}\big)(x) : x\in K\Big\}\\
%&=\Big\{x\Big(\psi\circ f \big(\psi^{-1}(x)\big)\Big) : x\in K\Big\}\\
%&=\Big\{\psi(y)\psi\big(f(y)\big) : y\in H\Big\}\\
=\Big\{\psi\big(yf(y)\big) : y\in H\Big\} =\psi\big(R_f(H)\big)=\psi(H) =K.
\end{align*}
Hence the map $\Psi$ makes sense and maps $F_H\big(\SG(\alpha)\big)$ bijectively into $F_K\big(S_{G'}(\alpha')\big)$ (its inverse is given by $g \mapsto \psi^{-1} \circ g \circ \psi$).

Similar to \cite[Proposition 2]{Amiri2017} one can show that the map $\Psi$ is a group homomorphism.
Since $\psi,\psi^{-1}\in C^{\infty}(G,G)$ are diffeomorphism, they are proper maps.
Therefore by Proposition \ref{top con} the map $C^\infty (G,G) \rightarrow C^\infty (G',G'), f\mapsto \psi \circ f\circ \psi^{-1}$ is continous.
By the same argument, the inverse of $\Psi$ is also continuous. Hence $\Psi$ is a topological isomorphism.
\end{proof}

We hope to investigate the geometry of this action and its stabiliser subgroups in future work.

\subsection*{The subgroup of bisections}\label{sect: BIS} \addcontentsline{toc}{subsection}{The subgroup of bisections}
The Lie group $\SG (\alpha)$ is naturally related to the group $\Bis (\cG)$ of bisections of the Lie groupoid $\cG$.
Let us first recall the definition of this group.

\begin{defn}[{Bisections of a Lie groupoid \cite[Definition 1.4.1]{Mackenzie05}}]\label{defn: bis}
A \emph{bisection} of a Lie groupoid $\cG = (G\toto M)$ is a smooth map $\sigma\colon  M\rightarrow G$ which satisfies
$$ \beta \circ \sigma = \id_M \text{ and } \alpha \circ \sigma \in \Diff (M).$$
The set of all bisections $\Bis(\cG)$ is a group with respect to the group operations
\begin{equation}\label{eq: bis mult}
 (\sigma_1\star\sigma_2)(x)=\sigma_1(x)\sigma_2\big((\alpha\circ\sigma_1)(x)\big) \text{ and }\sigma^{-1}=\Big(\sigma\circ(\alpha\circ\sigma)^{-1}(x)\Big)^{-1} \quad \forall x \in M.
\end{equation}
\end{defn}

\begin{rem}
 \begin{enumerate}
  \item Note that we use the symbol ``$\star$'' both for the multiplication in $\Bis (\cG)$ and in $\SG$.
However, there should be no room for confusion as it will always be clear from the setup and the signature of the mappings which product is meant.
\item Switching the r\^{o}les of $\alpha$ and $\beta$ in Definition \ref{defn: bis} we obtain another group of bisections which is isomorphic to $\Bis (\cG)$.
 It is thus just a matter of convention which group is studied and we chose to define $\Bis (\cG)$ here to fit to the definition of $\SG$ in Definition \ref{defn: SG} (cf.\ \cite{Mackenzie05,MR3351079} where the opposite convention is adopted).
 \end{enumerate}
\end{rem}

To see the relation between $\Bis (\cG)$ and $\SG (\alpha)$, we have to adapt the definition of the left-translations used in \cite[Definition 1.4.1]{Mackenzie05} to right-translations.

\begin{defn}\label{right trans}
A \emph{right-Translation} of the Lie groupoid $\cG = (G \toto M)$ is a pair $(\varphi , \varphi_0) \in \Diff (G) \times \Diff (M)$, such that $\alpha\circ\varphi=\varphi_0\circ\alpha$, $\beta\circ\varphi=\beta$, and each $\varphi_x\colon  G_x\rightarrow G_{\varphi_0(x)}$ is $R_g$ for some $g\in G^x_{\varphi_0(x)}$, where $R_g(y)=yg$.
\end{defn}

\begin{ex}[{\cite[p.22]{Mackenzie05}}]\label{bis-rt}
For $\sigma \in \Bis (\cG)$, we obtain a right-translation $(R_\sigma , \alpha \circ \sigma)$, where
$$R_\sigma\colon G\rightarrow G,\quad  R_\sigma(g) \coloneq g\sigma(\alpha(g)).$$
Conversely if $(\varphi, \varphi_0)$ is a right-translation on $\cG$, then $\sigma \colon M\rightarrow G$ with $\sigma(x)=x^{-1}\varphi(x)$ defines a bisection of $\cG$.
\end{ex}

\begin{lem}\label{bis1}
Let $\cG = (G \toto M)$ be a Lie groupoids.
Every right-translation $(\varphi,\varphi_0)$ on $\cG$ defines an element $\psi (x) \coloneq x^{-1} \varphi(x), x\in G$ of $\SG(\alpha)$ which is constant on $G_u$ for every $u\in M$.
Finally, the restriction of $\psi$ to $M$ is a bisection.
\end{lem}
\begin{proof}
Consider a right-translation $(\varphi, \varphi_0)$ and define $\psi\colon  G\rightarrow G$ with $\psi(x)=x^{-1}\varphi(x)$.
Then $\beta\big(\varphi(x)\big)=\beta(x)=\alpha(x^{-1})$ implies that $\psi$ makes sense.
Since the maps $\varphi$, inversion and multiplication of $\cG$ are smooth, $\psi$ is smooth.
Also $\beta\circ\psi=\alpha$, that is $\psi\in\SG$, in fact $\psi\in\SG(\alpha)$, since $R_{\psi}=\varphi$ is a diffeomorphism.

For the second part of the assertion, let $u\in M \subseteq G$ and $g\in G_u$.
We observe that the restriction of $\varphi$ on $G_{u}$ is of the form $R_{\theta}$ for some $\theta\in G$.
Hence $\psi (g) = g^{-1} \varphi(g) = g^{-1} g \theta = \theta$ and $\psi|_{G_u}$ is constant.

Finally, $(\beta\circ \psi)(u)=\beta\big(u^{-1}\varphi(u)\big)=\beta(u)=u$ for every $u\in M$, that is $\psi|_M$ is a right-inverse for $\beta$.
Also $(\alpha\circ \psi)(u)=\alpha\big(\varphi(u)\big)=\varphi_0\big(\alpha(u)\big)=\varphi_0(u)$, hence $\alpha\circ \psi|_M=\varphi_0$ is a diffeomorphism. Therefore $\psi|_M$ is a bisection.
\end{proof}

\begin{rem}
 Note that every diffeomorphism $\varphi \colon G \rightarrow G$ with $\beta\circ \varphi=\beta$, induces an element $R^{-1}(\varphi)\colon G\rightarrow G, \psi(x)=x^{-1}\varphi(x)$ of $\SG(\alpha)$.
 However, the restriction of $R(\varphi)$ on $M$ may not be a bisection, see Example \ref{ex: PG} below.
\end{rem}

%According to \cite[Theorem B]{HS2016}, the group $\Bis(\cG)$ with the fine very strong topology is a topological group.
In the following we show that $\Bis(\cG)$ is isomorphic to a subgroup of $\SG(\alpha)$.
%Furthermore, if $\cG$ is an $\alpha$-proper Lie groupoid, it turns out that $\Bis(G)$ is even isomorphic to a topological subgroup of $\SG(\alpha)$.

\begin{thm}\label{embedded}
Let $\cG = (G \toto M)$ be a Lie groupoid, then the map
\begin{equation}\label{eq: identBis}
 \Psi \colon \big(\Bis(\cG), \star\big) \rightarrow \big(\SG (\alpha),\star\big),\quad \sigma \mapsto \sigma \circ \alpha
\end{equation}
is a group monomorphism. If $\cG$ is in addition $\alpha$-proper, then
$\Psi$ is a morphism of Lie groups.   \label{thm: emb2}
\end{thm}

\begin{proof}
Let $\sigma\in \Bis(\cG)$, then we consider the induced right-translation $(R_{\sigma},\alpha\circ\sigma)$ (cf.\ Example \ref{bis-rt}).
By the Lemma \ref{bis1} the right-translation $(R_{\sigma},\alpha\circ\sigma)$ produces an element $\psi_{\sigma}\in\SG(\alpha)$ by
\begin{equation}\label{ind: sga}
\psi_\sigma(x)=x^{-1}R_{\sigma}(x)=x^{-1}x\sigma\big(\alpha(x)\big)=\sigma(\alpha\big(x)\big).
\end{equation}
Lemma \ref{bis1} states that $\psi_\sigma|_M$ is a bisection, and indeed $\psi_\sigma|_M =\sigma$.
Therefore every bisection of $\cG$ can be extended smoothly to yield an element of $C^{\infty}(G,G)$ which belongs to $\SG(\alpha)$.
Obviously the map $\Psi\colon \Bis(\cG)\rightarrow \SG(\alpha)$ with $\Psi(\sigma)=\psi_\sigma$ is injective. It is also a group homomorphism.
To check this, let $g\in G$, $\sigma, \tau \in \Bis (\cG)$ and compute using the formulae for the multiplications in $\Bis(\cG)$ and $\SG(\alpha)$,
\begin{align*}
\big(\Psi(\sigma\star\tau)\big)(g)&=\psi_{\sigma\star\tau}(g) \stackrel{\eqref{ind: sga}}{=}(\sigma\star\tau)\big(\alpha(g)\big) \stackrel{\eqref{eq: bis mult}}{=}\sigma\big(\alpha(g)\big)\tau\Big(\big(\alpha\circ\sigma\big)\big(\alpha(g)\big)\Big) \\
&=\psi_\sigma(g) \tau\Big(\alpha \big(\sigma(\alpha(g))\big)\Big) =\psi_\sigma(g)\tau\Big(\alpha \big(g\sigma\big(\alpha(g)\big)\big)\Big) =\psi_\sigma(g)\psi_\tau\big(g \psi_\sigma(g)\big)\\
&=(\psi_\sigma\star \psi_\tau)(g) =\big(\Psi(\sigma)\star\Psi(\tau)\big)(g).\end{align*}
Therefore $\Psi(\sigma\star\tau)=\Psi(\sigma)\star\Psi(\tau)$ and thus $\Psi$ is a group monomorphims identifying $\Bis (\cG)$ with a subgroup of $\SG (\alpha)$.

To prove the last assertion, assume now that $\cG$ is an $\alpha$-proper Lie groupoid, then $\alpha$ is proper, whence  \ref{setup: comp} shows that $\Psi$ is smooth as a mapping to $C^\infty_{fS}(M,G)$.
Following \cite[Proposition A]{AS18} (cf.\ \cite{MR3351079,HS2016}), $\Bis (\cG)$ is a submanifold of $C^\infty_{fS} (M,G)$, whence $\Psi$ is smooth as a mapping to $\Bis (\cG)$.
\end{proof}
% \begin{rem}
% Similar to Lemma \ref{con3} \ref{con3: 2} we can show that for an $\alpha$-proper Lie groupoid $\cG$ the map $\sigma\mapsto R_\sigma$ from $\Bis(G)$ onto $R_{Bis(G)}$ is topological isomorphism. Hence for $\alpha$-proper Lie groupoids, we have the following commutative diagram.
%
% \begin{displaymath}
%  \begin{xy}
%   \xymatrix{
%      Bis(\cG)_{\ar[d]^{\Psi\colon  \Psi(\sigma)=f_\sigma}} \ar[r]^-{\simeq} &R_{Bis(G)}=_{\ar[d]^{\Phi\colon  \Phi(R_\sigma)=R_{\Psi(\sigma)}}}\{ R_\sigma: \sigma\in Bis(G)\}\\
%      %\ar[d]^{\Psi}& \ar[d]^{\Phi}\\
%       \SG(\alpha) \ar[r]^-{\simeq} &  R_{\SG(\alpha)}=\{R_f :f\in \SG(\alpha)\}
%   }
% \end{xy}
%  \end{displaymath}
%
%
% \end{rem}

\begin{ex} (Pair groupoid)\label{ex: PG}
Let $M$ be a manifold and $G=M\times M$, then the Pair groupoid $P(M)$ is defined as follows: $\big((x,y),(z,w)\big)\in P(M)^{(2)}$ if and only if $z=y$ and $(x,y)(y,w)=(x,w)$, $(x,y)^{-1} = (y,x)$.
We have $\alpha(x,y)=y$ and $\beta(x,y)=x$ and units $(x,x), x\in M$.

For this Lie groupoid, the monoid $S_{P(M)}$ consists of the maps $$H\colon M\times M\rightarrow M\times M , \quad H(x,y)=\big(y,h(x,y)\big),\quad  h \in C^\infty (M\times M, M).$$
The group $S_{P(M)}(\alpha)$ consists of all $H\in S_{P(M)}$ where for every $x\in M$ the map $h(x,-)\colon M\rightarrow M, y\mapsto h(x,y)$ is a diffeomorphism.
It is easy to check that for $H\in S_{P(M)} (\alpha)$ the restriction of $H$ on $M$ is bisection if and only if the restriction of $h$ on the diagonal of $M\times M$ induces a diffeomorphism between $M$ and the diagonal.

If we take the submanifold $M=\mathbb{R} \setminus \{0\}$ of $\mathbb{R}$ and consider the pair groupoid $P(M) = (G \toto M)$, then $H\colon  G\times G \rightarrow G,(x,y) \mapsto (y,xy)$ is an element of $\SG(\alpha)$ whose its restriction $(x,x) \mapsto (x,x^2)$ on the base of $P(M)$ does not induce a diffeomorphism of the diagonal onto $M$. Thus $H$ does not restrict to a bisection.
\end{ex}

We have seen that the bisection group can be realised as a subgroup of $\SG (\alpha)$, though we do not know whether this group embedds in general into $\SG (\alpha)$ in a nice way.
However, since these two groups are related, we investigate in the next section the subset of elements which are contained in the graph of elements in $\SG$.

\section{Connecting subgroups of the unit group with Lie groupoids}\label{sect: subcon}

In this section we let $\cG = (G\toto M)$ be a Lie groupoid and we investigate subsets of $G \times G$ which arise as the union of graphs of elements in $\SG$.
Our interest here is twofold, in that we want to
\begin{enumerate}
 \item connect the monoid $\SG$ to the power set monoid of $G$,
 \item investigate the reconstruction of $\Gtwo$ (as a set) from graphs of $\SG$.
These questions are related to similar questions for (sub-)groups of bisections of the Lie groupoid $\cG$ (see e.g.\ \cite{ZHUO,MR3573833} and Remark \ref{rem: bis} below).
\end{enumerate}

However, as a first step let us establish a characterisation of elements in $\SG(\alpha)$ via their graphs.
Note that if $ f\in C^{\infty}(G,G)$, then $f\in \SG$ if and only if $\Gamma_f\subseteq \Gtwo$, where $\Gamma_f \coloneq \big\{\big(x,f(x)\big) \in G \times G \mid x \in G\big\}$ is the graph of $f$.
Remark \ref{G^2} shows that $(\Gtwo)^0 =\Gamma_{\alpha}$. It is well known, cf.\ \cite[p.\ 10]{Glofun} that the graph of a function $G \rightarrow G$ is a submanifold of $G\times G$, whence $\Gamma_f$ and $(\Gtwo)^0$ are submanifolds of $G\times G$ for every $f\in \SG$.
In the following we establish a new characterization of the elements of $\SG(\alpha)$.

\begin{prop}
Let $\cG = (G \toto M)$ be a Lie groupoid and $f\in \SG$.
Then $f\in \SG(\alpha)$ if and only if the submanifolds $\Gamma_f$ and $(\Gtwo)^0$ are diffeomorphic.
\end{prop}
\begin{proof} Consider the source map $\alpha^{2} \colon (\Gtwo) \rightarrow G \times M, (x,y) \mapsto \big(xy, \alpha (xy)\big)$ (cf.\ Remark \ref{G^2}). By \cite[Proposition 7]{Amiri2017}, $\alpha^2$ restricts to a bijection of $\Gamma_f$ to the base space of $\Gtwo$ if and only if $f\in \SG(\alpha)$.
Define
$$F\colon  (\Gtwo)^0\rightarrow \Gtwo, \quad F\big(x,\alpha(x)\big)=\Big(xg(x),f\big(xg(x)\big)\Big),$$ where $g\star f=\alpha$ (i.e.\ $g$ is the inverse of $f$ in $\big(\SG(\alpha), \star\big)$).
Obviously the image of $F$ is contained in $\Gamma_f$ and it is easy to check that $F$ is the inverse of $(\alpha^2)|_{\Gamma_f}$.
We deduce that $F$ is a bijection if $f \in \SG (\alpha)$.
The multiplication of $\cG$ and the maps $f, g$ are smooth, whence $(\alpha^2)|_{\Gamma_f}$ and its inverse $F$ are smooth.
In conclusion, $\Gamma_f$ is diffeomorphic to the base space of $\Gtwo$ which we identified with $(\Gtwo)^{0}$.
\end{proof}

\begin{cor}
The graphs of elements in $\SG(\alpha)$ are diffeomorphic.
\end{cor}

\begin{setup}[Relating $\SG$ and the power set monoid]\label{P(G)}
 For a groupoid $\cG = (G \toto M)$ we denote by $P(G)$ the powerset of $G$. Together with the multiplication of sets (for two subset $U, V$ of $G$, $U.V \coloneq \big\{xy \mid (x,y)\in (U\times V)\cap \Gtwo\big\}$) is a monoid.
 The base space $M$ is the unit of $P(G)$.
 Since $\Gtwo$ is again a groupoid, $P(\Gtwo)$ with the multiplication of sets is a monoid.
 If $f,g\in\SG$, then $\Gamma_f ,\Gamma_g\subseteq \Gtwo$. Therefore by the groupoid structure of $\Gtwo$ and by \cite[Proposition 8]{Amiri2017} $\Gamma_f.\Gamma_g=\Gamma_{f\star g}$.
\end{setup}
\begin{defn}
Let $\cG$ be a Lie groupoid. For every $A\subseteq\SG$, set $\Gamma(A)=\bigcup_{f\in A}\Gamma_f$.
 By the definition of $\SG$, $\Gamma(A)\subseteq \Gtwo$ for every $A\subseteq\SG$.
\end{defn}

\begin{lem}
If $A$ is a subgroup of $\SG(\alpha)$, then $\Gamma(A)$ forms a wide subgroupoid of $\Gtwo$.
\end{lem}
\begin{proof}
Note that $(G^2)^0=\big\{\big(x,\alpha(x)\big): x\in G\big\}=\Gamma_{\alpha}\subset \Gamma(A)$ whence the subset $\Gamma(A) \subseteq \Gtwo$ can only be made into a wide subgroupoid.
To prove that $\Gamma(A)$ is a subgroupoid of $\Gtwo$, it is enough to show that it is closed under multiplication and inverse.

Let $\gamma_1,\gamma_2\in \Gamma(A)$ and $(\gamma_1,\gamma_2)\in (\Gtwo)^{(2)}$. Then there exists $f_1, f_2\in A$ with $\gamma_1\in\Gamma_{f_1}, \gamma_2\in\Gamma_{f_2}$, hence $\gamma_1\gamma_2\in \Gamma_{f_1}.\Gamma_{f_2}=\Gamma_{f_1\star f_2}\subset\Gamma(A)$. Also if $\gamma=\big(x,f(x)\big)\in \Gamma(A)$ where $f\in A$, then $\gamma^{-1}=\big(xf(x),f(x)^{-1}\big)$ by the definition of inversion map on $\Gtwo$. Now if $g$ be the inverse of $f$ in $A$ and $y=xf(x)$, then by \cite[Proposition 3]{Amiri2017} $g(y)=(f(x))^{-1}$. Hence $\gamma^{-1} \in \Gamma_g\subset\Gamma(A)$.
\end{proof}

\begin{thm}\label{iso2}
Let $\cG = (G \toto M)$ and $\cG' = (G' \toto M)$ be two isomorphic Lie groupoids where $\psi \in \Diff (G,G')$ induces the isomorphism over the identity.
Let $H \toto M$ be a Lie subgroupoid of $\cG$ and $K\coloneq \psi(H)$, then $\psi \times \psi \colon G \times G \rightarrow G \times G$ induces a groupoid isomorphism
$$\Phi \colon \Gamma\Big(F_H\big(\SG(\alpha)\big)\Big) \rightarrow \Gamma\Big(F_K\big(S_{G'}(\alpha')\big)\Big),\quad \big(x,f(x)\big) \mapsto \big(\psi(x),\psi \circ f(x)\big).$$

\end{thm}
\begin{proof}
Let $\Psi\colon F_H\big(\SG(\alpha)\big)\rightarrow F_K\big(S_{G'}(\alpha')\big)$ be the group isomorphism constructed in Proposition \ref{Iso1}. Indeed $\Phi\big(x,f(x)\big)=\Big(\psi(x),\Psi(f)(\psi(x))\Big)$ for every $f\in F_H\big(\SG(\alpha)\big)$  and $x\in H$, whence $\Phi$ makes sense. Note that by construction $\Phi (\Gamma_f) = \Gamma_{\Psi (f)}$ for all $f \in F_H\big(\SG(\alpha)\big)$.
Clearly $\Phi$ is bijection (whose inverse is the restriction of $\psi^{-1} \times \psi^{-1}$).

If $\gamma=\big(x,f(x)\big)\in\Gamma_f$ and $\eta=\big(y,g(y)\big)\in\Gamma_g$ are two elements of
$\Gamma\Big(F_H\big(\SG(\alpha)\big)\Big)$ with $(\gamma,\eta)\in(H \times H)^{(2)}$, we have $\eta=\big(xf(x),g(xf(x))\big)$, $\Phi(\gamma)=\Big(\psi(x),\psi\big(f(x)\big)\Big)$
 and $\Phi(\eta)=\Big(\psi\big(xf(x)\big),\big(\Psi(g)\big)\big(\psi(xf(x))\big)\Big)=\Big(\psi\big(xf(x)\big),\psi\big(g(xf(x))\big)\Big)$.
 But $\psi$ and $\psi^{-1}$ are groupoid homomorphisms, whence $\psi(x)\psi\big(f(x)\big)=\psi\big(xf(x)\big)$ and $\psi^{-1} \big(xf(x)\big)= \psi^{-1} (x) \psi^{-1} \big(f(x)\big)$.
 Therefore $\big(\Phi(\gamma),\Phi(\eta)\big)\in (K\times K)^{(2)}$. Finally
\begin{align*}
 \Phi(\gamma)\Phi(\eta)&=\Big(\psi(x),\psi\big(f(x)\big)\Big)\Big(\psi\big(xf(x)\big),\psi\big(g(xf(x))\big)\Big)\\
 &=\Big(\psi(x),\psi\big(f(x)\big)\psi\big(g(xf(x))\big)\Big) =\Phi \Big(x,f(x)g\big(xf(x)\big)\Big)\\
 &= \Phi \big(x,f\star g(x)\big)
%  &=\Big(\psi(x),\big(\psi\circ (f\star g)\circ \psi^{-1}\big)(\psi(x))\Big)\\
%  &=\Big(\psi(x),\Psi(f\star g)(\psi(x))\Big)\\
%  &=\Phi\big(x, (f\star g)(x)\big)\\
%  &
=\Phi(\gamma\eta) \qedhere
 \end{align*}
\end{proof}

\begin{ex}
  Let $G$ be a Lie group and $\cG = (G \toto \ast)$, then as a set $\SG$ is equal to $C^{\infty}(G,G)$ and also $(\SG,\star)$ and $\big(C^{\infty}(G,G),\circ_\op\big)$ are two isomorphic monoids (cf.\ also Remark \ref{rem: gpbd}).

  Further, $\SG(\alpha)$ is isomorphic to $\Diff (G)$. Therefore in this case $\Gamma(\SG)=G\times G=\Gtwo$ and if $G$ is connected then the Homogeneity Lemma \cite[p.22]{MR0226651} implies that $\Gamma(\SG(\alpha))=G\times G=\Gtwo$.
   We note that $\SG(\alpha) \cap \Diff (G)$ consists of all smooth map $f\in C^{\infty}(G,G)$ for which $f$ and $R_f$ are diffeomorphisms. The authors do not know whether in general one has $\SG(\alpha)\cap \Diff(G)=\emptyset$ or not. Note however, that this set does not inherit a useful structure from either group (as it is easy to prove that for example the unit elements of both groups can not be contained in the intersection).
%    It is a challenge for the reader in the simple case  $G=\big((0,\infty),.\big)$ try  to characterize all elements of $Diff(G)\cap \SG(\alpha)$.
\end{ex}

\begin{rem}\label{rem: bis}
If $\sigma\in \Bis(\cG)$, then $\Gamma_{\sigma}\subseteq (M\times G)\cap \Gtwo$.
In \cite{ZHUO} (cf.\ also the later \cite[Proposition 2.7]{MR3573833}) the authors proved that if $G$ is a  $\beta-$connected Lie groupoid, then for every $g\in G$ there exists a bisection $\sigma\in \Bis(\cG)$ with $\sigma\big(\beta(g)\big)=g$, hence $\bigcup_{\sigma\in \Bis(\cG)}\Gamma_{\sigma}=(M\times G)\cap \Gtwo$.
\end{rem}

The previous remark leads to the following question.
\begin{ques}
For which Lie groupoids does one have $\Gamma(\SG)=\Gtwo$?
More general, is it possible to characterize the minimal subset $A$ of $\SG$ which satisfies in $\Gamma(A)=\Gtwo$?
If $\SG$ satisfies $\Gamma(\SG)=\Gtwo$, does this imply $\Gamma \big(\SG(\alpha)\big) = \Gtwo$?
\end{ques}
Consider the following two example with respect to the above question.
\begin{ex}
\begin{enumerate}
\item
It is well known (see e.g.\ \cite[Homogeneity Lemma, p.22]{MR0226651}) that if $M$ is a connected smooth manifold then for any two points $x, y\in M$ there exists a diffeomorphism $\phi\colon  M\rightarrow M$ such that $\phi(x)=y$. Therefore for the pair groupoid $P(M)$ of a connected smooth manifold $M$ for $\big((x_0,y_0)(y_0,z_0)\big)\in P(M)^{(2)}$ there exists a diffeomorphism $\phi_0\colon  M\rightarrow M$ such that $\phi_0(y_0)=z_0$.
Then Example \ref{ex: PG} shows that $F\colon M \times M\rightarrow M \times M$ with $F(x,y)=\big(y,\phi_0(y)\big)$ belongs to $S_{P(M)} (\alpha)$  and $\big((x_0,y_0)(y_0,z_0)\big)\in\Gamma_F$. Therefore $\Gamma\big(S_{P(M)}(\alpha)\big)=P(M)^{(2)}$.

Conversely one can argue as in \cite[Remark 2.18 b)]{MR3573833} to obtain: For a disconnected manifold $M$ where the connected componennts are not diffeomorphic, one has $\Gamma \big(S_{P(M)} (\alpha)\big) \neq P(M)^{(2)}$.
\item If a Lie group $H$ acts on a manifold $M$ and $G=H\times M$. We let $\cG_{H\ltimes M}$ be the associated action groupoid (where the partial multiplication is given by $(g,hx) (h,x) \coloneq (gh, x)$.
Then for every $h_0, k_0\in H$ there exists a diffeomorphism $\phi_0\in \Diff(H)$ with $\phi_0(h_0)=k_0$. Therefore let $\big((h_0,x_0)(k_0,y_0)\big)\in \cG_{H\ltimes M}^{(2)}$, if we define $F\colon  G\rightarrow G$ with $F(h,x)=\Big(\phi_0(h),\big(\phi_0(h)\big)^{-1}.x\Big)$ then it is easy to check that $F\in S_{\cG_{H\ltimes M}}$ and $\big((h_0,x_0)(k_0,y_0)\big)\in\Gamma_F$, hence $\Gamma(S_{\cG_{H \ltimes M}})=\cG_{H \ltimes M}^{(2})$. The authors do not know whether $\Gamma\big(S_{\cG_{H\ltimes M}}(\alpha)\big)=\cG_{H\ltimes M}^{(2)}$.
\end{enumerate}
\end{ex}

So far, the results on $\Gamma (A)$ have been purely algebraic. In general it is not clear whether $\Gamma(A)$ carries a submanifold structure.
However, one can prove the following.

\begin{lem}\label{closed}
If $A\subseteq\SG$ is a compact set, then $\Gamma(A)$ is a closed set in $\Gtwo$.
\end{lem}
\begin{proof}
Let $(x_{\gamma},y_{\gamma})_{\gamma\in I}$ be a net in $\Gamma(A)$ which convergence to $(x,y) \in \Gtwo$. Then for every $\gamma\in I$ there exists $f_\gamma\in A$ with $y_{\gamma}=f_{\gamma}(x_\gamma)$. Compactness of $A$ allows us to replace the net by a subnet which converges. Hence without loss of generality we can suppose that the net $(f_\gamma)$ is convergent in the fine very strong topology to a function $f\in A$. Since the fine very strong topology in finer that the compact-open topology, $(f_\gamma)$ uniformly converges to $f$ on compact subsets of $G$. Since the set $\{x_\gamma\}_{\gamma\in I}\cup\{x\}$ is compact, therefore $f_{\gamma}(x_{\gamma})$ converges to $f(x)$. Hence $y=f(x)$ and consequently $(x,y)\in\Gamma_f$.
\end{proof}
\begin{ques}
%We know that disjointness of two manifolds $M$ and $N$ is sufficient for that $M\cup N$ be a manifold but it is not necessary.
Is it possible to characterize all subsets $A$ of $\SG$ which $\Gamma(A)$ is a submanifold of the manifold $G\times G$?
\end{ques}

\section*{Acknowledgements}

This research was partially supported by the European Unions Horizon 2020 research and innovation programme under the Marie Sk\l{}odowska-Curie grant agreement No.\ 691070.
We thank D.M.\ Roberts (Adelaide) for interesting discussions leading to Lemma \ref{staceyroberts}.

\appendix

\section{Locally convex manifolds and spaces of smooth maps}\label{Appendix:
MFD}

In this appendix we collect the necessary background on the theory of manifolds
that are modelled on locally convex spaces and how spaces of smooth maps can be
equipped with such a structure. Let us first recall some basic facts concerning
differential calculus in locally convex spaces. We follow
\cite{hg2002a,BGN04}.

\begin{defn}\label{defn: deriv} Let $E, F$ be locally convex spaces, $U \subseteq E$ be an open subset,
$f \colon U \rightarrow F$ a map and $r \in \N_{0} \cup \{\infty\}$. If it
exists, we define for $(x,h) \in U \times E$ the directional derivative
$$df(x,h) \coloneq D_h f(x) \coloneq \lim_{t\rightarrow 0} t^{-1} \big(f(x+th) -f(x)\big).$$
We say that $f$ is $C^r$ if the iterated directional derivatives
\begin{displaymath}
d^{(k)}f (x,y_1,\ldots , y_k) \coloneq (D_{y_k} D_{y_{k-1}} \cdots D_{y_1}
f) (x)
\end{displaymath}
exist for all $k \in \N_0$ such that $k \leq r$, $x \in U$ and
$y_1,\ldots , y_k \in E$ and define continuous maps
$d^{(k)} f \colon U \times E^k \rightarrow F$. If $f$ is $C^\infty$ it is also
called smooth. We abbreviate $df \coloneq d^{(1)} f$ and for curves $c \colon I \rightarrow M$ on an interval $I$, we also write $\dot{c} (t) \coloneq \frac{\dd}{\dd t} c (t) \coloneq dc(t,1)$.
\end{defn}

\begin{defn}[Differentials
on non-open sets]\label{defn: nonopen}
\begin{enumerate}
\item A subset $U$ of a locally convex space $E$ is \emph{locally
convex} if every $x \in U$ has a convex neighbourhood $V$ in $U$.
\item Let $U\subseteq E$ be a locally convex subset with dense interior and $F$ a locally convex space.
A continuous mapping $f \colon U \rightarrow F$ is called $C^r$ if
$f|_{U^\circ} \colon U^\circ \rightarrow F$ is $C^r$ and each of the
maps $d^{(k)} (f|_{U^\circ}) \colon U^\circ \times E^k \rightarrow F$
admits a continuous extension
$d^{(k)}f \colon U \times E^k \rightarrow F$ (which is then necessarily
unique).
% Analogously, we say that a continuous map $g \colon U \rightarrow M$ to a smooth manifold $M$ is of class $C^r$ if the tangent maps
% $T^{k} (f|_{U^\circ}) \colon U^\circ \times E^{2^k-1} \rightarrow T^kM$ exist and
% admit a continuous extension $T^{k}f \colon U \times E^{2^k-1} \rightarrow T^kM$.
% Note that we defined $C^k$-mappings on locally convex sets with dense interior in two ways for topological vector spaces (when viewed as manifolds).
% However, by \cite[Lemma 1.14]{hg2002a} both conditions yield the same class of mappings.
% If $U \subseteq \R$ and $g$ is $C^{1}$, we obtain a continuous
% map $g' \colon U \rightarrow TM, g'(x) \coloneq T_x g(1)$. We shall
% write $\frac{\partial}{\partial x}g(x) \coloneq g' (x)$.
\end{enumerate}
\end{defn}

Note that for $C^r$-mappings in this sense the chain rule holds.
Hence there is an associated concept of locally
convex manifold, i.e., a Hausdorff space that is locally homeomorphic to open
subsets of locally convex spaces with smooth chart changes. See
\cite{neeb2006,hg2002a}
for more details.

\begin{setup}[$C^r$-Manifolds and $C^r$-mappings between them]
For $r \in \N_0 \cup \{\infty\}$, manifolds modelled on a fixed locally convex space can be defined as usual.
The model space of a locally convex manifold and the manifold as a topological space will always be assumed to be Hausdorff spaces.
However, we will not assume that infinite-dimensional manifolds are second countable or paracompact.
We say $M$ is a \emph{Banach} (or \emph{Fr\'echet}) manifold if all its modelling spaces are Banach (or
Fr\'echet) spaces.
\end{setup}
\begin{setup}
Direct products of locally convex manifolds, tangent spaces and tangent bundles as well as $C^r$-maps between manifolds may be defined as in the finite-dimensional setting.
%For $C^r$-manifolds $M,N$ we use the notation $C^r(M,N)$ for the set of all $C^r$-maps from $M$ to $N$.
%For every $f\in C^r(M,N)$ the graph of $f$ which is denoted by $\Gamma_f$ is a submanifold of $M\times N$.
Furthermore, we define \emph{locally convex Lie groups} as groups with a $C^\infty$-manifold structure turning the group operations into $C^\infty$-maps.
\end{setup}

 \section{Topology and manifold structures on spaces of smooth functions} \label{app: topo}

 In this section we recall some essentials concerning the topology on the space of smooth mappings $C^\infty (X,Y)$, where $X, Y$ are finite-dimensional (not necessarily compact) manifolds.
 It is well-known that on non-compact manifolds the compact-open $C^\infty$-topology is not strong enough to control the behaviour of smooth functions. Instead one has to use a finer topology which allows one to control the behaviour of functions on locally finite covers of the manifold by compact sets. To this end let us briefly recall the construction of this topology (cf.\ e.g.\ \cite{michor1980,HS2016}).

 \begin{setup}[Conventions]
 In the following $X$, $Y$ and $Z$ denote smooth paracompact and finite-dimensional manifolds.
  For a multiindex $\{i_1, \ldots i_n\}$ and $f \colon \mathbb{R}^n \supseteq U \rightarrow \mathbb{R}^m$ (with $U \subseteq \mathbb{R}^n$ a locally convex subset with dense interior) sufficiently differentiable, we denote by $\frac{\partial}{\partial x_{i_1} \ldots \partial x_{i_n}} f$ the iterated partial derivative (where the $x_i$ are derivation in $i$th coordinate direction).
 \end{setup}

 \begin{defn}[Elementary neighborhood]
	\label{elementarynbh}
	Let $r \in \N_0$ and $f \colon X \to Y$ smooth, $(U,\phi)$ a chart on $X$, $ (V,\psi) $ a chart on $ Y $, $ A \subseteq U $ compact such that $ f(A) \subseteq V $, and $ \epsilon > 0 $. Define the set $\mathcal{N}^r \big(f; A,(U,\phi),(V,\psi),\epsilon \big)$ as
	\begin{align*}
	\left\lbrace \vphantom{\left\Vert \frac{\partial^j}{\partial x_{s_1} \ldots \partial x_{s_j}}\right.} h \in C^\infty (M,E) \middle|\right. & \quad  h(A) \subseteq V, \\ &\left.\sup_{a \in A} \sup_{0 \leq j \leq r} \left\Vert \frac{\partial^j}{\partial x_{s_1} \ldots \partial x_{s_j}} \left(\psi \circ h \circ \phi^{-1} - \psi \circ f \circ \phi^{-1}\right) (a) \right\Vert < \epsilon \right\rbrace .
	\end{align*}
	We call this set an \emph{elementary ~$C^r$-neighborhood of }~$ f $ in ~$ C^\infty (X,Y) $.
\end{defn}

\begin{defn}[Basic neighborhood]
	\label{basicnbh}
	Let $ f \colon X \to Y $ be a smooth map. A \emph{basic neighborhood of $f$ in ~$ C^\infty (X,Y) $} is a set of the form
	\begin{displaymath}
	\bigcap_{i \in \Lambda } \mathcal{N}^{r_i} \big(f; A_i,(U_i, \phi_i),(V_i,\psi_i),\epsilon_i\big),
	\end{displaymath}
	where $ \Lambda $ is a possibly infinite indexing set, for all \( i \) the other parameters are as in Definition \ref{elementarynbh}, and $ \lbrace A_i \rbrace_{i \in \Lambda} $ is locally finite.
\end{defn}

 \begin{defn}[Very strong topology]
	The \emph{very strong topology on $ C^\infty (X,Y) $} is the topology with basis the basic neighborhoods in $ C^\infty (X,Y) $ and we denote the space $C^\infty (X,Y)$ with the very strong topology will by $C^\infty_{vS}(X,Y)$.
\end{defn}

One can prove that many natural operations performed on mapping spaces become continuous with respect to the very strong topology.
However, if one wants to consider $C^\infty (X,Y)$ as an infinite-dimensional manifold a stronger topology is needed.
This is achieved by considering the following refinement.

\begin{defn}[The fine very strong topology]
	Define an equivalence relation $ \sim $ on $ C^\infty (X,Y) $ by declaring that
	$$ f \sim g \quad :\Leftrightarrow \quad \overline{\big\lbrace y \in M \mid f(y) \neq g(y) \big\rbrace} \text{ is compact } .$$ Now refine the very strong topology on $ C^\infty (X,Y) $ by demanding that the equivalence classes are open in $ C^\infty (X,Y) $. In other words, equip $ C^\infty (X,Y) $ with the topology generated by the very strong topology and the equivalence classes. This is the \emph{fine very strong topology} on $ C^\infty (X,Y) $. We write $ C^\infty_{fS}(X,Y) $ for $ C^\infty (X,Y) $ equipped with the fine very strong topology.
\end{defn}

\begin{rem}
 In \cite{michor1980} the very strong topology is called $\mathcal{D}$-topology and the fine very strong topology is called $\cF\cD$-topology.
 One can prove that the $\cF\cD$-topology turns $C^\infty (X,Y)$ into a smooth (infinite-dimensional) manifold.
\end{rem}

Foremost our interest is in the following facts concerning the continuity of the composition of mappings with respect to the topology.

\begin{defn}
 Recall that a $f\in C^{\infty}(X,Y)$ is \emph{proper} if and only if the preimage of a compact set $K$ in $Y$ is compact in $X$.
 Further the set $\Prop(X, Y)\subseteq C^{\infty}(X,Y)$ is open in the (fine) very strong topology.
\end{defn}

\begin{prop}[{\cite[Theorem 2.5 and Proposition 2.7]{HS2016}}]\label{top con}
Endow $C^\infty (X,Y)$ and $\Prop (X,Y)$ either with the very strong topology or the fine very strong topology.
 Let $f \colon Y \rightarrow Z$ be a smooth map, then
 $$  f_* \colon  C^{\infty}(X,Y)\rightarrow C^{\infty}(X,Z),\ \ h\mapsto f\circ h \quad \text{ for } f \in C^\infty (Y,Z) $$
 is continuous. Further, the composition map
 \begin{align*}
  \mathrm{Comp}\colon  C^{\infty}(Y,Z) \times \Prop(X, Y )\rightarrow C^{\infty}(X, Z) , \quad ( f , g)\mapsto f\circ g
 \end{align*}
 is continuous.
\end{prop}

\begin{cor}\label{cor: ev}
Endow $C^\infty (X,Y)$ either with the very strong or fine very strong topology. Then the evaluation $\ev \colon C^\infty (X,Y) \times X \rightarrow Y, (f,x) \mapsto f(x)$ is continuous.
\end{cor}

\begin{proof}
 Since the one-point manifold $\ast$ is compact $\Prop (\ast , X) = C^\infty (\ast,X) \cong X$. Now $\ev (f,x)= \mathrm{Comp} (f, \ast \mapsto x)$ and the continuity follows from Proposition \ref{top con}.
\end{proof}

\subsection*{The manifold structure}
It is well known that with respect to the fine very strong topology the space of mappings $C^\infty (X,Y)$ becomes an infinite-dimensional manifold.
Let us briefly recall the ingredients of this construction.

For a (possibly infinite-dimensional) manifold $N$ we denote by $\mathbf{0} \colon N \rightarrow TN$ the zero-section into the tangent bundle.

\begin{defn}[Local Addition]
Let $N$ be a (possibly infinite-dimensional) manifold.  A \emph{local addition} on $N$ is a smooth map $\Sigma \colon TN \rightarrow N$ together with an open neighborhood $\Omega \subseteq TN$ of the zero-section, such that
\begin{enumerate}
 \item  $(\pi_{TN}, \Sigma) \colon \Omega \rightarrow U \subseteq N\times N$ induces a diffeomorphism onto an open neighborhood $U$ of the diagonal in $N \times N$,
  \item $\Sigma \circ \mathbf{0} = \id_N$.
\end{enumerate}
If $N$ is finite dimensional, a local addition can always be constructed as the exponential map associated to a Riemannian metric.
\end{defn}

\begin{setup}[Smooth structure of $C^\infty (X,Y)$]\label{setup: smoothstruct}
Consider $f \colon X \rightarrow Y$ and define
  \begin{align*}
   U_f &\coloneq \big\{h \in C^\infty (X, Y) \mid  h \sim f, (f(x),h(x)) \in U,\  \forall x \in X\big\} \\
   \mathcal{D}_f (X, TY) &\coloneq \big\{h \in C^\infty (X,TY) \mid \pi_{Y} \circ h =f , \text{ and } h \sim \mathbf{0} \circ f\big\}
  \end{align*}
  where $\pi_Y \colon TY \rightarrow Y$ is the bundle projection.

Then we obtain a manifold chart for $C^\infty (X,Y)$ around $f \in C^\infty (X,Y)$ by defining
\begin{align*}
 \varphi_f \colon U_f \rightarrow \mathcal{D}_f (X,TY), \quad  h \mapsto (\pi_{Y},\Sigma)^{-1} \circ (f,h).
\end{align*}
Note that $ \mathcal{D}_f (X, TY)$ is canonically isomorphic to the locally convex space $\Gamma_c (f^* TY)$ of compactly supported sections with values in the pullback bundle $f^*TY$ (cf.\ \cite[1.17]{michor1980}).
 For later reference we note that the inverse of the chart is given by the formula $\varphi^{-1}_f (h) = \Sigma \circ h$.
 These charts turn $C^\infty_{fS} (X,Y)$ into an infinite-dimensional manifold, \cite[Theorem 10.4]{michor1980}. In the following we will always endow spaces of the form $C^\infty_{fS} (X,Y)$ with this manifold structure.
 Recall from \cite[Proof of Theorem 10.4]{michor1980} that the manifold structure of $C^\infty_{fS} (X,Y)$ does not depend on the choice of local addition used in the construction.
\end{setup}

The smooth structure induced by the fine very strong topology turns the continuous mappings discussed in the previous section into smooth maps.
Namely we have the following:

\begin{setup}[Pushforward by smooth maps] \label{setup: pf}
Let $\theta \colon Y \rightarrow Z$ be a smooth map.
Then \cite[Corollary 10.14]{michor1980} establishes smoothness of the pushforward
 $$\theta_* \colon C^\infty_{fS} (X,Y) \rightarrow C^\infty (X,Z),\quad f \mapsto \theta \circ f.$$ 
\end{setup}

\begin{setup}[{Smoothness of composition \cite[Theorem 11.4]{michor1980}}] \label{setup: comp}
Endowing $\Prop (X,Y) \subseteq C^\infty_{fS} (X,Y)$ with the subspace topology, the composition map
 \begin{align*}
  \mathrm{Comp}\colon  C^{\infty}(Y,Z) \times \Prop(X, Y )\rightarrow C^{\infty}(X, Z) , \quad ( f , g)\mapsto f\circ g\\
 \end{align*}
 is smooth. In particular,
 $\ev \colon C^\infty_{fS} (X,Y) \times X \rightarrow Y ,\quad  (f,x) \mapsto f(x)$ is smooth.
\end{setup}

\section{Details for the proof of the Stacey-Roberts Lemma}\label{app: detailedproofs}
 \begin{setup}[Conventions]
  For the rest of this section, we let $X,Y$ be finite-dimensional, paracompact manifolds and $\theta \colon X \rightarrow Y$ be a smooth surjective submersion.
 \end{setup}

 \begin{lem}[{Extracted from \cite[Proof of Theorem 5.1]{1301.5493v1}}]\label{lem: VFsubm}
  There is a smooth map $X_\theta \colon TX \rightarrow \mathfrak{X}_c (X)$ into the compactly supported vector fields on $X$ such that
  \begin{align}
   X_\theta (v)\big(\pi (v)\big) &= v \quad \forall v \in TX, \ \pi \colon TX \rightarrow X \text{ bundle projection} \label{VF: eq 1}\\
   T\theta \big(X_\theta (v)(x)\big) &= 0 \quad \text{ if } v \in \mathrm{Ker}\, T\theta \label{VF: eq 2}
  \end{align}
 \end{lem}

 \begin{proof}
  For brevity we set $d \coloneq \text{dim } X$ and $k \coloneq \text{dim } Y$.

  \paragraph{Step 1:} \emph{Preparing charts and a variant partition of unity}
  Since $\theta$ is a submersion, we find submersion charts for each $p \in X$, i.e.\ manifold charts
  $\varphi_p \colon X \supseteq U_p \rightarrow \R^d = \R^{d-k} \times \R^{k}$ and $\psi_{\theta(p)} \colon Y \supseteq U_{\theta (p)} \rightarrow \R^k$ such that
  \begin{itemize}
   \item $\theta(U_p) \subseteq U_{\theta (p)}$, $\varphi_p (p) = 0$ and $\psi_{\theta (p)} \big(\theta (p)\big) = 0$,
   \item $\psi_{\theta (p)}^{-1} \circ \theta \circ \varphi_p^{-1} = \text{pr}_2$, where $\text{pr}_2 \colon \R^d = \R^{d-k}\times \R^k \rightarrow \R^{k}, (x,y) \mapsto y$.
  \end{itemize}
  By construction the chart domains $\{U_p\}$ cover $X$, whence we can construct a variant of a partition of unity subordinate to the covering $\{U_p\}$.
  Since $X$ is paracompact and locally compact, we can choose (see \cite[II, \S 3, Proposition 3.2]{langdgeo2001}) a locally finite subcover $\{W_i\}_{i\in I}$ of relatively compact open sets\footnote{A set is relatively compact if its closure is compact} subordinate to the cover $\{U_p\}$. We choose and fix for each $i \in I$ a $p_i \in X$ with $\overline{W}_i \subseteq U_{p_i}$.
  A trivial variation of the usual construction of a partition of unity (cf.\ e.g.\ \cite[II, \S 3, Theorem 3.3]{langdgeo2001}) we obtain a family of smooth functions $\{\rho_i \colon X \rightarrow \R, i \in I\}$ such that
  \begin{enumerate}
   \item The supports of the $\rho_i$ are compact and contained in $W_i \subseteq U_{p_i}$.
   \item for each $x \in X$ we have $\sum \big(\rho_i (x)\big)^2 = 1$.
  \end{enumerate}

  \paragraph{Step 2:} \emph{The maps $X_i$}
  Recall that in the chart $\varphi_{p_i} \colon U_{p_i} \rightarrow \R^d$ induces diffeomorphisms $TU_{p_i} \cong \R^d \times \R^d$ and $\mathfrak{X} (U_{p_i}) \rightarrow C^\infty (\R^d, \R^d), Y \mapsto \text{pr}_2 \circ T\varphi_{p_i} \circ Y\circ \varphi_{p_i}^{-1}$. Using these identifications we define a map $X_{p_i} \colon TU_{p_i} \rightarrow \mathfrak{X} (U_{p_i})$ as the map which corresponds to $\R^d \times \R^d \rightarrow C^\infty (\R^d, \R^d), (q,v) \mapsto (x \mapsto v)$.
  Note that $X_{p_i}$ is \textbf{not} smooth (not even continuous!). We will remedy this problem by using our variant partition of unity.
  However, let us record the following properties of $X_{p_i}$ first
  \begin{equation}\begin{aligned}
   X_{p_i} (v)\big(\pi (v)\big) &= v \quad  \forall v \in TU_{p_i} \text{ and}\\ \text{if } T\theta (v) = 0 \text{ then } T\theta \big(X_{p_i} (v)(x)\big)&=0 \quad \forall v \in TU_{p_i}, x \in U_{p_i}.
   \end{aligned} \label{eq: PXPI}
  \end{equation}

  Define for $i\in I$ a mapping $X_i \colon TX \rightarrow \mathfrak{X}_c (X)$ via the formula
  \begin{displaymath}
   X_i (v) (q) \coloneq \begin{cases}
                         \rho_i \big(\pi(v)\big) \rho_i (q) X_{p_i} (v)(q)& \text{if} \pi (v) ,q \in W_i \\
                         0 & \text{otherwise}
                        \end{cases}.
  \end{displaymath}
  Observe first that for each $v \in TM$ the map $X_i (v)$ is smooth with compact support contained in $\supp \rho_i$.
  In particular, $X_i$ takes its values in $\mathfrak{X}_c (X)$ and is smooth by Lemma \ref{lem: aux:idim} below, where we use $\mathfrak{X}_c (U_i) \cong C^\infty_c (\R^d, \R^d)$ induced by $\varphi_{p_i}$.
  Further, the construction ensures that the properties \eqref{eq: PXPI} are inherited by $X_i$.

  \paragraph{Step 3:} \emph{The map $X_\theta$.}
  Set $X_\theta \colon TX \rightarrow \mathfrak{X}_c (X),\ v \mapsto \sum_{i \in I} X_i (v)$. Since $X_i(v) \in \mathfrak{X}_c (X)$ and for a given $v$ only finitely many of the $X_i$ are non-zero, the map $X_\theta$ makes sense.

  To establish continuity of $X_\theta$ observe that on the closed set $\pi^{-1} (\overline{W}_i)$ the map $X_\theta$ is given by a finite sum of maps $X_i$ (since every $\overline{W}_i$ is compact and the supports of the $X_i$ form a locally finite cover). Thus $X_\theta|_{\pi^{-1} (\overline{W}_i)}$ is continuous and smooth on $\pi^{-1}(W_i)$ as a finite sum of such mappings (using that $\mathfrak{X}_c (X)$ is a locally convex space).
  As the family $\{\overline{W}_i\}_{i \in I}$ is locally finite, this proves that $X_\theta$ is continuous.
  Now $X_\theta$ is smooth as it restricts to a smooth map on each member of the open cover $\{\pi^{-1} (W_i)\}_{i\in I}$.

  Unraveling the definition of $X_\theta$, properties \eqref{VF: eq 1} and \eqref{VF: eq 2} now follow directly from the construction and \eqref{eq: PXPI} (cf.\ \cite[p.\ 29]{1301.5493v1}).
 \end{proof}

 \begin{lem}\label{lem: aux:idim}
Let $M$ be a (possibly infinite-dimensional) manifold. Consider a smooth map $f \colon M \times X \rightarrow \R^n$ with $n\in \N$ such that $f$ vanishes outside $M\times K$ for $K \subseteq X$ compact.
Then $f^\wedge \colon M \rightarrow C^\infty_c (X,\R^n), m \mapsto f(m,\cdot)$ is smooth.
 \end{lem}

 \begin{proof}
  Every map $f(m,\cdot) \colon X \rightarrow \R^n$ vanishes outside of $K$, whence $f^\wedge$ takes its image in $C^\infty_K (X,\R^n) \coloneq \big\{h \in C^\infty (X,\R^n)\mid \supp h \subseteq K\big\} \subseteq C^\infty_c (X, \R^n)$.

  Since $C^\infty_K (X,\R^n)$ is a closed subspace of $C^\infty_c (X,\R^n)$ with the fine very strong topology, it suffices to prove that $f^\wedge$ is smooth as a mapping into $C^\infty_K (X,\R^n)$.
  Now it is easy to see (cf.\ \cite[Remark 4.5]{HS2016}) that the subspace topology on $C^\infty_K (X,\R^n)$ coincides with the topology induced by the compact open $C^\infty$ topology (which is for $X$ non-compact strictly coarser than the fine very strong topology on $C^\infty (X,\R^n)$.

  Choose a locally finite family $\{V_\alpha\}_{\alpha \in I}$ of relatively compact sets which cover $X$.
  Then \cite[Proposition 8.13 (a) and (b)]{hg2004} asserts that the map
  $$\Gamma \colon C^\infty_K (X, \R^n) \rightarrow \bigoplus_{\alpha \in I, V_\alpha \cap K \neq \emptyset} C^\infty (V_\alpha , \R^n)_{c.o.},\quad f \mapsto (f|_{V_\alpha})_{\alpha \in I}$$
  is a topological embedding with closed image, where the spaces on the right side are endowed with the compact open $C^\infty$-topology.
  As $K$ is compact, only finitely many $\alpha \in I$ satisfy $V_\alpha \cap K \neq \emptyset$.
  As finite direct sums and products coincide, the map $\Gamma \circ f^\wedge \colon M \rightarrow \prod_{\alpha \in I, V_\alpha \cap K \neq \emptyset} C^\infty (V_\alpha, \R^n)$ is smooth if the components $M \rightarrow C^\infty (V_\alpha, \R^n)$ are smooth.
  However, by the exponential law \cite[Theorem A]{alas2012} this is the case since $f|_{M \times V_\alpha} \colon M \times V_\alpha \rightarrow \R^n$ is smooth.
 \end{proof}

 \begin{lem}[{Extracted from \cite[Proof of Theorem 5.1]{1301.5493v1}}]\label{lem: adaptedlocadds}
  Denote by $\mathcal{V} \subseteq TX$ the vertical subbundle given fibre-wise by $\mathrm{Ker}\, T_p \theta$.
  There exists a smooth horizontal distribution $\mathcal{H} \subseteq TX$ (i.e.\ a smooth subbundle such that $T X = \mathcal{V} \oplus \mathcal{H}$) and local additions $\eta_X$ on $X$ and $\eta_Y$ on $Y$ such that the following diagram commutes:
 \begin{equation}\label{eq: diag} \begin{aligned}
 \begin{xy}
  \xymatrix{
     TX =  \mathcal{V} \oplus\mathcal{H} \ar[d]^{0 \oplus T\theta|_{\mathcal{H}}} &  \ar[l]_-{\supseteq}\Omega_X \ar[r]^{\eta_X}  & X \ar[d]^{\theta}  \\
      TY & \ar[l]_-{\supseteq} \Omega_Y \ar[r]^{\eta_Y}  &   Y
  }
\end{xy}\end{aligned}
 \end{equation}
 \end{lem}

 \begin{proof}
 As a first step, we turn $\theta$ into a Riemannian submersion with respect to some auxiliary Riemannian metrics.
 Recall from \cite[Lemma 2.1]{dHF15} that one can construct a $\theta$-transverse Riemannian metric $G_t$ on $X$ as follows:
 Choose an Ehresmann connection for $\mathcal{V}$, i.e.\ a smooth horizontal distribution $\mathcal{H}$ complementing $\mathcal{V}$.
 Then we declare $\mathcal{V}$ and $\mathcal{H}$ to be orthogonal and choose a Riemannian metric $G_Y$ on $Y$ which induces a metric on $\mathcal{H}$ via pullback by $T\theta$.
 For $\mathcal{V}$ we choose now an arbitrary bundle metric $G_\mathcal{V}$and combining the pieces we obtain a Riemannian metric $G_t$ which is $\theta$-transverse and, by construction (cf.\ \cite[Section 2.1]{dHF15}), turns $\theta$ into a Riemannian submersion.

 The Riemannian metrics $G_t$ and $G_Y$ give rise to Riemannian exponential maps $\exp_X \colon \Omega_X \rightarrow X$ and $\exp_Y \colon \Omega_Y \rightarrow Y$.
 Here $\Omega_X$ and $\Omega_Y$ are suitable neighborhoods of the zero section such that $T\theta (\Omega_X) \subseteq \Omega_Y$ and the mappings $(\pi , \exp_X)$ and $(\pi_Y, \exp_Y)$ with $\pi,\pi_Y$ the canonical bundle projections induce diffeomorphisms onto a neighborhood of the diagonal in $X\times X$ (or $Y\times Y$).
 In other words $\exp_X$ and $\exp_Y$ induce local additions and we set $\eta_Y \coloneq \exp_Y$.

 The Riemannian submersion $\theta$ maps horizontal geodesics (i.e.\ geodesics with initial value in $\mathcal{H}$) to geodesics in $Y$, whence we obtain the formula.
 %As geodesics in $Y$ (locally) lift to horizontal geodesics (cf.\ \cite[Proposition 2.109]{MR2088027}) whence horizontal geodesics locally coincide with geodesics, we obtain for $x \in X$ and $v \in \mathcal{H}_n \cap \Omega_N$ the formula
 \begin{equation}\label{eq: exp:comm}
   \theta \circ \exp_X (v) = \eta_Y \circ T \theta (v) \quad v \in \Omega_X \cap \mathcal{H}
 \end{equation}
 Our aim is now to modify $\exp_X$ to obtain a local addition $\eta_X$ such that \eqref{eq: exp:comm} holds for each $v \in \Omega_X$.
 To this end, we need to construct several auxiliary maps:

 \paragraph{Step 1:} \emph{Parallel transport in $\mathcal{V}$.} Denote for a smooth map $\gamma \colon [0,1] \rightarrow X$ the parallel transport along $\gamma$ with respect to the vertical distribution\footnote{We work here with parallel transport with respect to the bundle metric on $\mathcal{V}$ and we may assume that parallel transport above any smooth curve exists (since we can choose a suitable metric on $\mathcal{V}$, cf.\ \cite[17.9 Theorem]{MR2428390}). The point is to obtain a smooth transport which stays in $\mathcal{V}$.} by $P_\gamma^{t_0,t_1} \colon T_{\gamma (t_0)} X \rightarrow T_{\gamma (t_1)} X$.
 From the exponential law \cite[Theorem 7.8 (d)]{MR3351079}, we see that the assignment $\alpha \colon TX \rightarrow C^\infty ([0,1], X), \ v \mapsto \alpha_v$ is smooth, where $\alpha_v(t) \coloneq \exp_X (tv)$.
 Now as parallel transport arises as the flow of an ordinary differential equation which depends smoothly on the parameters \cite[17.8 Theorem]{MR2428390}, we obtain a smooth map
 $$P \colon TX = \mathcal{V}\oplus \mathcal{H}\rightarrow \mathcal{V} \subseteq TX, \quad (v_{\mathcal{V}}, v_{\mathcal{H}}) \mapsto P^{0,1}_{\alpha_{v_{\mathcal{H}}}} (v_{\mathcal{V}}). $$

 \paragraph{Step 2:} \emph{A $\theta$-adapted local addition.}
  We consider now the smooth map $X_\theta \colon TX \rightarrow \mathfrak{X}_c (X)$ constructed in Lemma \ref{lem: VFsubm}.
  Then let $\Fl_1 \colon \mathfrak{X}_c (X) \rightarrow \Diff_c (X)$ the mapping which takes a vector field to its flow evaluated at time $1$.
  Recall that $\Fl_1$ is the Lie group exponential map (\cite[4.6]{MR702720} or cf.\ \cite[Theorem 5.4.11]{MR3328452} together with \cite[Corollary 13.7]{hg2015}) for the Lie group $\Diff_c (X)$, whence it is smooth.
  Denote by $\text{ev} \colon C^\infty (X,X) \times X \rightarrow X, (\varphi,x) \mapsto \varphi(x)$ the evaluation which is smooth by \cite[Corollary 11.7]{michor1980}.
  We then obtain a smooth map
  $$\eta_1 \colon TX \rightarrow X,\quad v \mapsto \text{ev} \big(\Fl_1 (X_\theta (v)),\pi(v)\big).$$
  Since $T\theta X_\theta (v) = 0$ for every $v \in \mathcal{V}$ we see that $\theta$ is constant on the integral curves of $X_\theta (v)$, i.e.\
  \begin{equation}\label{eq: theta:const}
   \theta \circ \eta_1|_{\mathcal{V}} = \theta \circ \pi_X|_{\mathcal{V}}
  \end{equation}
  Further, we note that $X_\theta (0_x) = 0 \in \mathfrak{X}_c (X)$, whence $\eta_1 (0_x) = x$ for every $x \in X$.
  For later use let us show that $\eta_1$ restricts to a local addition on some $0$-neighborhood.
  To this end we prove that $T_{0_x}(\pi_X , \eta_1 ) \colon T_{0_x} TX \rightarrow T(X \times X)$ has invertible differential for every $x \in X$.
  Recall that $T_0 \Fl_1 = \id_{\mathfrak{X}_c (X)}$ since $\Fl_1$ is the Lie group exponential of $\Diff_c (X)$. Then the chain rule and \eqref{VF: eq 1} show that
  \begin{align*}
   T_{0_x} \text{ev} \circ (\Fl_1 \circ X_\theta , \pi) &= T \text{ev} \circ (T_0 \Fl_1 (T_{0_x} X_\theta) , T_{0_x} \pi) =  T \text{ev} \circ (T_{0_x} X_\theta , T_{0_x} \pi)\\
							&= T_{0_x} (\text{ev} \circ (X_\theta , \pi)) = T_{0_x} \id_{TX},
  \end{align*}
  whence the derivative of $(\pi_{X} , \eta_1)$ over $0_x$ is (up to identification)\footnote{For the identification recall from \cite[X \S 4, Theorem 4.3]{langdgeo2001} that the zero section $\mathbf{0} \colon X \rightarrow TX$ induces a canonical isomorphism $\mathbf{0}^* (T^2X) \cong TX \oplus TX$ given in local coordinates by $(x,0,v,w) \mapsto ((x,v) , (x,w))$. Further the diagonal map $\Delta \colon X \rightarrow X \times X$ induces the identification $\Delta^* (T (X\times X)) \cong TX \oplus TX$ via the local formula $((x,x) (v,w)) \mapsto ((x,v) , (x,w))$.} given by the matrix $\begin{bmatrix} \id & 0 \\ \id & \id \end{bmatrix}$.
  Applying the inverse function theorem, $\eta_1$ restricts to a local addition on some neighborhood of the zero-section.
  One can show that this local addition is adapted to $\theta$ in the sense of \cite[Definition 3.1]{MR3351079} but we will not need this here.

  \paragraph{Step 3:} \emph{Construction of the map $\eta_X$.}
  We can now use the local addition $\eta_1$ to modify $\exp_X$ as follows:
  \begin{displaymath}
   \eta_X \colon \Omega_X \cap \mathcal{V} \oplus \mathcal{H} \rightarrow X , \quad (v_{\mathcal{V}}, v_{\mathcal{H}}) \mapsto \eta_1 \big(P (v_{\mathcal{V}}, v_{\mathcal{H}})\big)
  \end{displaymath}
  By construction, $\eta_X$ is a smooth map and combining \eqref{eq: theta:const} and \eqref{eq: exp:comm} we see that for $(v_{\mathcal{V}}, v_{\mathcal{H}}) \in \Omega_X$ we have
  \begin{displaymath}
   \theta \eta_X (v_{\mathcal{V}}, v_{\mathcal{H}}) = \theta \big(\alpha_{v_\mathcal{H}}(1)\big) = \theta \circ \exp_{X} (v_\mathcal{H}) = \exp_Y \big(\theta (v_\mathcal{H})\big) =  \eta_Y \circ T\theta (v_{\mathcal{V}}, v_{\mathcal{H}}).
  \end{displaymath}
  Thus $\eta_X$ satisfies \eqref{eq: diag} and we are left to prove that $\eta_X$ indeed restricts to a local addition on some neighborhood of the zero section.

  \paragraph{Step 4:} \emph{$\eta_X$ induces a local addition.} The old argument was deleted in this version and can either be seen in the earlier arXiv versions or the published version.

  \paragraph{Note added in 2024:} 
  \textbf{Step 4 in the published and the earlier  arXiv version contains a critical error}: The last tangential map in the argument was erroneously calculated. Moreover, the proof idea for Step 4 was flawed and will not work as originally presented. I thank Pelle Steffens for bringing this to my attention and discussing with me the solution of this problem.
  
  Note however, that the statement of the Stacey Roberts Lemma and everything up to Step 4 is correct though.
  Fortunately, Pelle Steffens was so kind as to supply a correct argument in his article \cite[Lemma 3.2.18]{ste24} (or go directly to the arXiv version \url{https://arxiv.org/abs/2404.07931}).

  The cited Lemma 3.2.18 is a version of the Stacey-Roberts Lemma which also contains the proofs of the Lemmata contained in the present paper as well as equivalent arguments for Steps 1-3 in the proof. 
  The argument replacing Step 4 can be found on p.50 of the arXiv version, it is the remaining part of the proof of Lemma 3.2.18 after the first commutative diagram on p.50.
 \end{proof}
\addcontentsline{toc}{section}{References}
\bibliography{Amiri_Schmeding_Differentiable_Monoid}

\end{document}